\documentclass[12pt]{article}

\usepackage[a4paper]{geometry}
\usepackage{amssymb}
\usepackage{amsmath}
\usepackage{amsthm}

\usepackage[dvipdfmx]{graphicx}
\theoremstyle{definition}
\newtheorem{theorem}{Theorem}[section]
\newtheorem{proposition}[theorem]{Proposition}
\newtheorem{lemma}[theorem]{Lemma}

\newtheorem{definition}[theorem]{Definition}

\newtheorem{assumption}[theorem]{Assumption}
\newtheorem*{definition*}{Definition}
\newtheorem*{example*}{Example}
\newtheorem*{problem*}{Problem}
\newtheorem*{problems*}{Problems}

\theoremstyle{remark}
\newtheorem{remark}[theorem]{Remark}
\newtheorem*{remark*}{Remark}
\newtheorem{remarks}[theorem]{Remarks}
\newtheorem*{remarks*}{Remarks}

\numberwithin{equation}{section}

\title{Exponential growth and decay estimates for eigenfunctions with complex eigenvalues}

\author{Kenichi {\scshape Ito}\footnote{Department of Mathematics, Graduate School of Science, Kobe University,
1-1, Rokkodai, Nada-ku, Kobe 657-8501, Japan.
E-mail: \texttt{ito-ken@math.kobe-u.ac.jp}. 
}
}

\date{}

\begin{document}

\allowdisplaybreaks
\maketitle

\begin{abstract}
We prove a priori exponential growth and decay estimates for locally square integrable eigenfunctions 
with complex eigenvalues of the non-self-adjoint Schr\"odinger operator. 
There appear two natural critical exponents, 
given by the imaginary parts of square roots of twice an eigenvalue.
Any associated eigenfunction has the same exponential growth or 
decay rate as one of these, 
or greater than the upper one. 
The results may be seen as extensions
of the Phragm\'en--Lindel\"of theorem, the Agmon estimates and Rellich's theorem to complex eigenvalues. 
For the proofs we adapt a commutator scheme of Ito--Skibsted~(2020) to 
the non-self-adjoint framework, 
presenting a unified approach to all of them. 
In order to control contributions from the imaginary parts of a potential and an eigenvalue, 
an additional commutator estimate is employed. 
\end{abstract}



\tableofcontents

\section{Settings and results}\label{sec:intr}

\subsection{Introduction}

In this paper, we discuss growth rate of locally $L^2$ eigenfunctions 
associated with complex eigenvalues of the non-self-adjoint Schr\"odigner operator 
\[H=-\tfrac12\Delta+V \]
on $\mathbb R^d$ with $d\in\mathbb N=\{1,2,\ldots\}$. 
Here $\Delta$ is the ordinary Laplacian, and we shall often write it as 
\begin{equation*}
-\Delta=p\cdot p=p_jp_j
;\quad 
p_j=-\mathrm i\partial_j\ \ \text{for }j=1,\dots,d ,
\end{equation*}
adopting Einstein's summation convention. 
The potential $V$ is complex-valued.
While the real part of $V$ is of long-range type, 
the imaginary part is of long- or middle-range type depending on purposes. 
See Assumption~\ref{cond:20022317} below for the precise assumptions. 

The goal of the paper is to determine exponential growth or decay rate of 
an eigenfunction $\phi\in L^2_{\mathrm{loc}}(\mathbb R^d)$ solving  
\begin{equation}
(H-\kappa^2)\phi=0\ \ \text{in the distributional sense} 
\label{200125}
\end{equation}
with 
$\kappa=\lambda+\mathrm i\mu\in\mathbb C_+=\{\kappa\in\mathbb C;\ \mathop{\mathrm{Im}}\kappa\ge 0\}$. 
It is in fact a \textit{generalized} eigenfunction in the sense that it may not be in $L^2(\mathbb R^d)$,
but we shall simply call it an eigenfunction. 
Throughout the paper we consistently call $\kappa$ a spectral parameter, 
and distinguish it from the associated eigenvalue $\kappa^2$. 
Then, our main results say that $\phi$ can only grow like 
\[
\log |\phi|=-\sqrt2\mu |x|+\mathcal O(|x|^{1-\rho})
,
\quad  \text{or}\quad  
\log |\phi|\ge \sqrt2\mu |x|+\mathcal O(|x|^{1-\rho})
.
\]
The absence of eigenfunctions between the upper and the lower exponential functions 
has been known for negative eigenvalues as the Phragm\'en--Lindel\"of theorem 
or the Agmon estimates~\cite{MR168901,MR745286}. 
In the following let us simply call it the Agmon estimates. 
On the other hand, that below the lower exponential function
has been known for positive eigenvalues 
as Rellich's theorem~\cite{MR17816,MR466902}. 
This paper extends both of them to all the complex eigenvalues with sharp subexponential error terms, 
and moreover provides a unified approach, 
based on a relatively recent commutator method of Ito--Skibsted~\cite{MR4062329}.

In the proofs we shall particularly focus on the following three operators, or quadratic forms 
\begin{align}
\mathop{\mathrm{Re}}\bigl(\Theta(H-\kappa^2)\bigr)
&=\tfrac12\bigl(\Theta(H-\kappa^2)+(H-\kappa^2)^*\Theta\bigr)
,
\label{260515}
\\
\mathop{\mathrm{Im}}\bigl(\Theta(H-\kappa^2)\bigr)
&=\tfrac1{2\mathrm i}\bigl(\Theta(H-\kappa^2)-(H-\kappa^2)^*\Theta\bigr)
,\label{260516}
\\
\mathop{\mathrm{Im}}\bigl(A\Theta(H-\kappa^2)\bigr)
&=\tfrac1{2\mathrm i}\bigl(A\Theta(H-\kappa^2)-(H-\kappa^2)^*\Theta A\bigr)
.
\label{260517}
\end{align}
Here $A$ is a certain self-adjoint operator, called \textit{conjugate operator}, 
and $\Theta$ is a multiplication operator by a real-valued weight function. 
If $V$ and $\kappa^2$ are real, 
they are indeed (anti-)commutators (with weight $\Theta$ inside), 
and hence \eqref{260515}--\eqref{260517} are their natural generalizations. 
It is effective to keep the form $H-\kappa^2$, rather than just $H$, 
since it annihilates eigenfunctions, 
so that these forms are zero on them. 
However, if we implement the commutators on the right-hand sides, 
two of these quantities, \eqref{260515} 
and \eqref{260517}, turn out to be 
\textit{weightedly} positive or negative definite due to characteristics of the system,  
with admissible errors. 
Therefore, we can bound the leading term by the error terms, and a desired estimate follows. 
This is in fact how a commutator technique in general works, 
and we follow such a strategy with regard to \eqref{260515} and \eqref{260517}. 
Meanwhile, the second quantity \eqref{260516} 
is an important key in bounding non-self-adjoint contributions from $V$ and $\kappa^2$, 
which appear as first order differential operators. 
Although they are generally unbounded,
we have established a computational scheme to control them with \eqref{260516}.

The commutator method of Ito--Skibsted~\cite{MR4062329} 
originates in the Mourre theory~\cite{MR603501,MR1388037}. 
In an attempt to remove errors of compact operators and energy cut-offs, 
with which there appear a countable number of exceptional points in the spectrum, 
they have reached a combination 
of a generator $A$ of radial translations and carefully designed weight functions $\Theta$. 
Recall that Mourre's conjugate operator is the generator of dilations. 
Although their commutator method requires long computations, 
it finds a wide range of applications 
in the spectral theory of the Schr\"odinger operator, being elementary and optimal. 
In fact, it was already applied to an exterior problem on a manifold 
with Euclidean and/or hyperbolic ends~\cite{MR4062329}, 
the inverted harmonic oscillator~\cite{MR3936114,MR4316933}, 
constant electric fields~\cite{MR4066044,MR4465249}, 
$N$-body potentials~\cite{MR4274369}, 
some class of long-range perturbations~\cite{MR4754869},
the harmonic oscillator~\cite{MR4926373},
and slowly decaying attractive potentials~\cite{ito2026lowenergyresolventestimates}, 
as well as to the Dirac operator~\cite{arita2026elementarycommutatormethoddirac}. 
As in these applications, we expect that 
the method of the paper would also be useful in deducing exponential resolvent estimates 
for $\kappa^2\in \mathbb C\setminus [0,\infty)$,
which we would like to discuss elsewhere.   

Finally, let us mention that, after the main part of the paper has been done, 
the author becomes aware that our method  
shares common features with the \textit{Morawetz multiplier method}. 
Compare \eqref{260515}--\eqref{260517} with Perthame--Vega~\cite[(2.1)--(2.3)]{MR1695559}, 
and see also \cite{MR2289541,MR4977402}, though their goals are different from ours. 
An advantage of our method is that 
we have carefully and maximally exploited choice of weight functions, 
effectively combining spatial cut-offs and the Yosida approximation. 
We also remark that our commutator method may sometimes be 
called the \textit{Carleman estimates}, or the \textit{virial argument}. 
However, we would prefer to keep the terminology ``commutator'' with its origin in mind.

\subsection{Settings}

\subsubsection{Modified radius}

Throughout the paper we prefer to use a \textit{modified radius function} $f$ to
avoid the exact radius $|x|$ possessing a singularity of at the origin. 
Choose $\chi\in C^\infty(\mathbb{R})$ such that 
\begin{equation}
\chi(t)
=\left\{\begin{array}{ll}
1 &\mbox{ for } t \le 1, \\
0 &\mbox{ for } t \ge 2,
\end{array}
\right.
\quad
\chi'\le 0,
\label{eq:14.1.7.23.24}
\end{equation}
and define $f\in C^\infty(\mathbb R^d)$ as 
\begin{equation*}
f(x)=
\chi(|x|)+|x|(1-\chi(|x|))\ \ \text{for }x\in\mathbb R^d.
%\label{260503}
\end{equation*}
Obviously, $f$ coincides with $|x|$ for $|x|\ge 2$, but the values for $|x|<2$ are modified so that 
$f$ is uniformly positive on the whole $\mathbb R^d$. 
We shall also use $f$ as a weight function, 
replacing a rather standard one $\langle x\rangle=(1+x^2)^{1/2}$. 

\subsubsection{Radial and spherical derivatives}

With the above modified radius $f$, we define the radial differential operator $A$ as 
\begin{equation}
A=\mathop{\mathrm{Re}}p_f=\tfrac12(p_f+p_f^*)
;\quad 
p_f=-\mathrm i\partial_f
,\ \ 
\partial_f=(\partial_jf)\partial_j
.
\label{260521}
\end{equation}
It is non-singular and globally defined on $\mathbb R^d$, 
and we will adopt it as a conjugate operator in our commutator theory 
along with various weight functions. 
Note that it coincides with a generator of the unitary radial translations
outside a compact subset of $\mathbb R^d$.  
In addition, we introduce a spherical differential operator $L$ as 
\begin{equation}
L=p_j\ell_{jk}p_k
;\quad 
\ell_{jk}=\delta_{jk}-(\partial_jf)(\partial_kf)\ \ \text{for }j,k=1,\dots,d
.
\label{25710b}
\end{equation}
Then we have a useful representation for $H$, given as   
\begin{equation}
H=\tfrac12A^2+\tfrac12L+q
;\quad 
q=V+\tfrac18(\Delta f)^2+\tfrac14(\partial_f\Delta f).
\label{260504}
\end{equation}
Note that we have the same formula even for $d=1$ just by regarding $L=0$.

\subsubsection{Complex potentials}

Now we present precise assumptions on the complex potential $V$. 
Since we will always work with the representation \eqref{260504}, 
we would like to impose assumptions on $q$. 
Restatement in terms of $V$ is straightforward as the difference is \textit{short-range}.

\begin{assumption}\label{cond:20022317}
The perturbation $q$ splits as 
\[
q=(q_1+q_2)+\mathrm i(r_1+r_2);
\quad 
q_1,r_1\in C^1(\mathbb R^d;\mathbb R),\ \ 
q_2,r_2\in L^\infty(\mathbb R^d;\mathbb R)
,
\]
and one of the following holds.  
\begin{enumerate}
\renewcommand{\labelenumi}{(\alph{enumi})}
\item\label{2606021}
There exist $\rho\in(0,1)$ and $C>0$ such that on $\mathbb R^d$
\begin{align*}
q_1= 0,\qquad 
|q_2|\le Cf^{-\rho},\qquad 
r_1= 0,\qquad  
|r_2|\le Cf^{-\rho}.
\end{align*}
\item\label{2606022}
There exist $\rho\in(0,1)$ and $C>0$ such that on $\mathbb R^d$
\begin{align*}
q_1&\le Cf^{-\rho},&
\partial_fq_1&\le Cf^{-1-\rho},& 
|q_2|&\le Cf^{-1-\rho},&
r_1&=0,&
|r_2|&\le Cf^{-(1+\rho)/2}.
\end{align*}
\item\label{2606023}
There exist $\rho\in(0,1)$ and $C>0$ such that on $\mathbb R^d$
\begin{align*}
q_1&\le Cf^{-\rho},&
\partial_fq_1&\le Cf^{-1-\rho},& 
|q_2|&\le Cf^{-1-\rho},\\
|r_1|&\le Cf^{-(1+\rho)/2},&
|\partial r_1|&\le Cf^{-1-\rho}, &
|r_2|&\le Cf^{-1-\rho}.
\end{align*}
\item\label{2606024}
There exist $\rho\in(0,1)$ and $C>0$ such that on $\mathbb R^d$
\begin{align*}
q_1&\le Cf^{-2\rho},&
\partial_fq_1&\le Cf^{-1-2\rho},& 
|q_2|&\le Cf^{-1-2\rho},\\
|r_1|&\le Cf^{-(1+3\rho)/2},&
|\partial r_1|&\le Cf^{-1-3\rho}, &
|r_2|&\le Cf^{-1-2\rho}.
\end{align*}
\end{enumerate}
\end{assumption}
\begin{remarks}\label{260704}
\begin{enumerate}
\item 
A perturbation of order $\mathcal O(f^{-\epsilon})$ or $\mathcal O(f^{-1-\epsilon})$
for some $\epsilon>0$ is said to be of long- or short-range type, respectively. 
Thus we may say that a perturbation of order $\mathcal O(f^{-1/2-\epsilon})$ is of middle-range type. 
The critical exponent $1/2$ also appears in 
connection with the \textit{Landis conjecture}~\cite{MR723095,MR1669361}. 
See the assertion 2 of Proposition~\ref{20260535017b} for a related result. 
As for the conjecture itself, we refer to 
\cite{MR1110071,fernandezbertolin2024landisconjecturesurvey,MR4929648}. 
\item 
We will work only outside of a large compact subset of $\mathbb R^d$,
and hence can add compactly supported singular perturbations.  
Then, for the second main result, 
we have to assume the \textit{unique continuation property}, i.e., 
if $(H-\kappa^2)\phi=0$ on $\mathbb R^d$ and $\phi=0$ in some open subset of $\mathbb R^d$,
then $\phi=0$ on $\mathbb R^d$. 
Note that this property holds for a wide class of perturbations~\cite{WO}, and is not that 
restrictive. 
\end{enumerate}
\end{remarks}

\subsubsection{Exponentially weighted $L^2$ spaces} 

Finally, before stating our main results, let us introduce several weighted $L^2$ spaces. 
We defined for any $s,t\in\mathbb R$
\[
\mathcal E_{s,t}=\mathrm e^{sf+tf^{1-\rho}}L^2(\mathbb R^d)
, 
\]
where $\rho$ is from Assumption~\ref{cond:20022317}. 
In addition, we define auxiliary spaces 
\begin{align*}
\mathcal E_{s+0}&=\bigcap_{r>s}\mathrm e^{rf}L^2(\mathbb R^d)
,&
\mathcal E_{s-0}&=\bigcup_{r<s}\mathrm e^{rf}L^2(\mathbb R^d)
,&
\mathcal E_{-\infty}&=\bigcap_{s\in\mathbb R}\mathrm e^{sf}L^2(\mathbb R^d)
,
\end{align*}
and 
\begin{align*}
\mathcal E_{s,t+0}&=\bigcap_{r>t}\mathrm e^{sf+rf^{1-\rho}}L^2(\mathbb R^d)
,&
\mathcal E_{s,-\infty}&=\bigcap_{t\in\mathbb R}\mathrm e^{sf+tf^{1-\rho}}L^2(\mathbb R^d)
\end{align*}
Note that here and below
we adopt the same signs for subscripts and exponents, conversely to the ordinary convention, 
since this makes it much easier to see growth rate of a function belonging to the space.

\subsection{Main results}

\subsubsection{Agmon estimates}

The first main result of the paper is 
an extension of the Agmon estimates
\cite{MR168901,MR745286} to complex eigenvalues. 
It says the absence of eigenfunctions between the upper and the lower exponential functions,
and may be seen as, simultaneously, a lower bound for the upper eigenfunctions 
and an upper bound for the lower eigenfunctions. 

\begin{theorem}\label{260603}
Suppose Assumption~\ref{cond:20022317} with (a), 
let $\kappa=\lambda+\mathrm i\mu\in \mathbb C_+\setminus\mathbb R$,
and let $t\ge 1$ be sufficiently large. 
Then, if $\phi\in  \mathcal E_{\sqrt2\mu,-t}$ solves \eqref{200125}, 
one has $\phi\in \mathcal E_{-\sqrt2\mu,t}$. 
%In addition, if there exists $\gamma\ge 0$ such that 
%\[\langle \phi,L\rangle\le \gamma\langle\phi, f^{-1}\phi\rangle,\]
%then one has $\phi\in(\mathcal B^*)_{\sqrt2\mu,-(1+\gamma)/2}$. 
\end{theorem}

\begin{remark}
Theorem~\ref{260603} is sharp in the sense that for any $t\in\mathbb R$
we can construct potentials $V_\pm$ with eigenfunctions $\phi_\pm$ obeying 
\[
\phi_+\in \mathcal E_{\sqrt2\mu,- t+0}\setminus\mathcal E_{\sqrt2\mu,- t}
,\quad 
\phi_-\in \mathcal E_{-\sqrt2\mu,t+0}\setminus\mathcal E_{-\sqrt2\mu,t}
,
\]
respectively, see Section~\ref{260613}. 
We note that, if we kept track of the constants appearing in the proof, 
we would have an explicit bound for $t$ in terms of $q$ and $\kappa$.
However, we shall not elaborate it in the paper. 
\end{remark}

Our proof is directly motivated by Ito--Skibsted~\cite{MR4062329} 
and Tagawa~\cite{MR4926373}. 
We will show that the quantities \eqref{260515} and \eqref{260517} 
tend to be positive or negative on the eigenstate $\phi$ 
for appropriate weight functions $\Theta$. 
One of the technical novelties of the paper is to bound non-self-adjoint contributions 
from $q$ and $\kappa^2$ by using the commutator \eqref{260516}. 
Although there is an extensive amount of literature on this topic, 
the author is still unaware of a work discussing complex eigenvalues 

Here let us briefly review a part of previous research on exponential upper bounds for 
the lower eigenfunctions. 
In its very early stage the main focus seems to have been 
on $N$-body potentials~\cite{MR471726,MR516993}. 
Among others, the Froese--Herbst method \cite{MR682117,MR723095} was established along this line, 
see also a review paper by Hislop~\cite{MR1785381}. 
Then, while Agmon introduced the Agmon metric~\cite{MR745286}, 
such exponential decay estimates prove effective in the analysis of 
the tunneling effect in the semiclassical regime, 
see a book by Dimassi--Sj\"ostrand~\cite{MR1735654} and references therein. 
As for recent development, we refer to works on higher order elliptic operators~\cite{MR1675240,MR3286533,MR3581301},
and on the discrete Schr\"odinger operator~\cite{MR2453259,MR4583551}.

\subsubsection{Rellich's theorem}

The second main result of the paper is an extension of 
Rellich's theorem~\cite{MR17816,MR466902} to complex eigenvalues,
saying the absence of eigenfunctions below the lower exponential function. 
It may also be seen as a lower bound for the lower eigenfunctions.

\begin{theorem}\label{2606034}
Suppose Assumption~\ref{cond:20022317} with (b), (c) or (d)
according to $\kappa=\lambda+\mathrm i\mu\in \mathbb C_+\setminus \mathbb R$, 
$\mathbb R\setminus \{0\}$ or $\{0\}$, respectively, 
and let $t\ge 1$ be sufficiently large. 
Then, if $\phi\in \mathcal E_{-\sqrt2\mu,-t}$ solves \eqref{200125}, 
one has $\phi= 0$ on $\mathbb R^d$. 
\end{theorem}

\begin{remarks}\label{260705}
\begin{enumerate}
\item
Similarly to Theorem~\ref{260603}, 
the function class in Theorem~\ref{2606034} is sharp, see Section~\ref{260613}.
It would be also possible to bound $t$ in terms of $q$ and $\kappa$, 
but we shall not elaborate it. 
\item
If the imaginary part of $q$ is of short-range type and $\kappa\in\mathbb R\setminus\{0\}$, 
then we can replace $\mathcal E_{-\sqrt2\mu,-t}$ by the \textit{Agmon--H\"ormander space} 
$\mathcal B^*_0$ as in Ito--Skibsted~\cite{MR4062329}.
\item
The case $\kappa=0$ is quite subtle, where the result depends 
not only on the decay rate of $q$ but also on its sign. 
In fact, if $q$ is \textit{slowly decaying} and \textit{attractive}
in the sense that $-Cf^{-\sigma}\le q\le  -cf^{-\sigma}$ for some $c,C>0$ and $\sigma\in(0,2)$, 
then zero-eigenfunctions are absent in a much larger space $f^{(2+\sigma)/4}L^2(\mathbb R^d)$,
see, e.g., \cite{ito2026lowenergyresolventestimates}. 
The result of Theorem~\ref{2606034} for $\kappa=0$
is the best possible one without imposing any assumptions on sign of $q$. 
\end{enumerate}
\end{remarks}

In the proof of Theorem~\ref{2606034}, 
in contrast to Theorem~\ref{260603}, we will only make use of positivity of \eqref{260517}. 
The basic strategy is similar to Ito--Skibsted~\cite{MR4062329},
but we have to further use \eqref{260515} again to control non-self-adjoint contributions.
Rellich's theorem for complex eigenvalues seem to have rarely been considered so far, 
and Theorem~\ref{2606034} would be new. 

Let us mention some of relevant literature. 
Rellich's theorem was first established for the unperturbed Laplacian by Rellich~\cite{MR17816}. 
For the Schr\"odinger operator the absence of $L^2$ eigenfunctions 
with positive eigenvalues was obtained by \cite{MR108633,MR192186,MR252875,MR601072}, 
see also a book by Reed--Simon~\cite[Kato--Agmon--Simon theorem]{MR493421}. 
On the other hand, Uchiyama~\cite{MR372407} and Mochizuki~\cite{MR417553} 
discussed the same problem in some larger weighted $L^2$ spaces,
and Agmon--H\"ormander \cite{MR466902} in an optimally large Besov-type space,
or the Agmon--H\"ormander space $\mathcal B^*_0$. 
The last result was further generalized to the $N$-body Schr\"odinger operator 
by Isozaki~\cite{MR1272984} and Adachi--Itakura--Ito--Skibsted~\cite{MR4274369}. 
We remark that the method of Ito--Skibsted~\cite{MR4062329} was affected 
by the Froese--Herbst method~\cite{MR687957,MR682117,MR723095}, 
which was originally developed for the absence of $L^2$ eigenfunctions. 
Note that the last two works \cite{MR682117,MR723095} also discuss the Agmon estimates.

\subsubsection{Sharpness of the results, and discussion}\label{260613}

Let us present an elementary example that 
verifies the sharpness of Theorems~\ref{260603} and \ref{2606034}.
For simplicity we consider the one-dimensional case with $d=1$. 
For any $\rho\in(0,1)$, 
$\kappa=\lambda+\mathrm i\mu\in\mathbb C_+$ and $\tau=t+\mathrm ir\in\mathbb C$ 
set 
\begin{equation*}
\phi=\mathrm e^{\pm\mathrm i\sqrt2\kappa f+\tau f^{1-\rho}}
\in
\mathcal E_{\mp\mu,t+0}
\setminus\mathcal E_{\mp\mu,t}
,
%\label{260601317}
\end{equation*}
and then they in fact satisfy 
\[
(H-\kappa^2)\phi=0
\]
with 
\begin{equation}
V=V_\pm
=\pm \mathrm i\sqrt2 (1-\rho)\kappa \tau f^{-\rho}
+\tfrac12(1-\rho)^2\tau^2f^{-2\rho}
-\tfrac12\rho(1-\rho)\tau f^{-1-\rho}
,
\label{260601318}
\end{equation}
respectively. 
Now we can easily see the sharpness of Theorems~\ref{260603} and \ref{2606034}. 
Let us note that for any $\kappa\in\mathbb R\setminus\{0\}$ 
the subexponential exponent vanishes if and only 
if the imaginary part of the potential \eqref{260601318} is of short-range type,
which comforms with 2 of  Remarks~\ref{260705}. 
In addition, if $\kappa=0$, the leading term of the potential 
\eqref{260601318} is of order 
$\mathcal O(f^{-2\rho})$, and this explains the doubled exponents 
in Assumption~\ref{cond:20022317} (d).

Another natural question would be whether it is possible to verify absence of 
generalized eigenfunctions \textit{above} the upper exponential function. 
However, it would be false in general for $d\ge 2$. 
For example, in $\mathbb R^2$, set for any $\alpha,\beta,\gamma,\delta\in\mathbb R$ 
\[
\phi(x,y)=\mathrm e^{\mathrm i\sqrt2(\alpha+\mathrm i\beta)x+\mathrm i\sqrt2(\gamma+\mathrm i\delta)y}
\in
\mathcal E_{\sqrt{2(\beta^2+\delta^2)}+}
\setminus\mathcal E_{\sqrt{2(\beta^2+\delta^2)}}
.
\] 
Then it is an eigenfunction for the free Schr\"odinger operator 
$H_0=-\tfrac12\Delta$, and the associated eigenvalue $E$ is given by 
\begin{equation}
E
=\alpha^2-\beta^2+\gamma^2-\delta^2+2\mathrm i(\alpha\beta+\gamma\delta)
.
\label{26061322}
\end{equation}
Therefore, for any given $E\in\mathbb C$, we can find $\alpha,\beta,\gamma,\delta\in\mathbb R$
satisfying \eqref{26061322} with arbitrarily large $\sqrt{2(\beta^2+\delta^2)}$. 

Nonetheless, if we impose an additional assumption on
an expectation value of the angular momentum $L$ from \eqref{25710b} on the eigenstate $\phi$, 
then we would have the same exponential upper bound for $\phi$, like Tagawa~\cite{MR4926373}. 
This can be seen from the proof of Theorem~\ref{2606034}. 
We just have to invert inequalities appearing there to show \textit{negativity} of \eqref{260517}, 
except for the term concerning $L$. 

Finally, we remark that it would be easy to generalize 
our results to an exterior domain or to a manifold 
with Euclidean and/or hyperbolic funnel ends
by following the formulation of Ito--Skibsted~\cite{MR4062329}. 
Our arguments are dependent only on a few properties of an \textit{escape function} $f$, 
such as convexity,
and no other global geometry of the space $\mathbb R^d$ is required. 
It would be also interesting whether it is possible to adopt another escape function,
e.g., a solution to a complex-valued eikonal equation,  
to remove a subexpoential term from the exponent, and to allow long-range perturbations. 
Hopefully, we would be able to discuss it elsewhere.

\section{Preliminaries}

In this section we discuss two topics as preliminaries for the proofs of Theorems~\ref{260603}
and \ref{2606034}. 
One is commutator formulas involving a general weight function $\Theta$,
and the other is the \textit{refined Yosida approximation} $\theta$. 
Both of the proofs of Theorems~\ref{260603} and \ref{2606034} 
are dependent on sequential use of commutator arguments, 
and we shall step by step proceed to stronger estimates. 
In each of these steps we have to change $\Theta$ appropriately, 
and hence general formulas for an unspecified $\Theta$ would be very useful. 
Meanwhile, in most of these steps, 
we will have to deduce an additional exponential decay for an eigenfunction, 
and for that purpose the so-called Yosida approximation is well suited. 
We will introduce a refined version of it in accordance with the exponent $\rho\in (0,1)$ 
from Assumption~\ref{cond:20022317}. 

\subsection{General commutator formulas}

Here let us present several commutator formulas for later reference. 
Let us impose only the following least requirements on $\Theta$.

\begin{assumption}\label{26060220}
Let $\Theta\in C^\infty(\mathbb R^d;\mathbb R)$ be a function only of $f$, and satisfy
\begin{align*}
\mathop{\mathrm{supp}}\Theta\subset \{x\in\mathbb R^d;\ |x|\ge 2\}. 
\end{align*}
\end{assumption}
We note that under Assumption~\ref{26060220} 
we have $f=|x|$ on $\mathop{\mathrm{supp}}\Theta$, so that 
\begin{gather*}
%\begin{gathered}
|\partial f|=1
,\quad 
\ell_{ij}\ell_{jk}=\ell_{ik}
,\quad 
\partial_i\partial_jf=f^{-1}\ell_{ij}
,\\ 
(\partial_if)(\partial_i\partial_jf)=0
,\quad 
\Delta f=(d-1)f^{-1}
,\quad 
Lf=0
%\end{gathered}
%\label{eq:220319b}
\end{gather*}
on the same subset. 
These formulas are very basic, and we shall often use them without reference.
In addition, if we denote the derivatives of $\Theta$ in $f$ by primes, then 
on $\mathop{\mathrm{supp}}\Theta$ we have 
\[
\partial_f\Theta=\Theta',\quad 
\partial_f^2\Theta=\Theta'',\quad 
\dots,\quad  
\partial_f^k\Theta=\Theta^{(k)}\ \ \text{for }k\in\mathbb N_0.  
\]
Now we present the following general commutator formulas.

\begin{lemma}
Suppose Assumption~\ref{26060220}. 
Then for any $\kappa=\lambda+\mathrm i\mu\in\mathbb C_+$ one has 
\begin{align}
\begin{split}
\mathop{\mathrm{Re}}\bigl(\Theta(H-\kappa^2)\bigr)
&=
\tfrac12A\Theta A
+\tfrac12\Theta L
-(\lambda^2-\mu^2)\Theta 
-\tfrac14 \Theta''
+(q_1+q_2)\Theta 
,
\end{split}
\label{26052222}
\\
\begin{split}
\mathop{\mathrm{Im}}\bigl(\Theta(H-\kappa^2)\bigr)
&=
\tfrac12\mathop{\mathrm{Re}}(\Theta'A)
-2\lambda\mu\Theta
+(r_1+r_2)\Theta 
,
\end{split}
\label{26052223}
\end{align}
and 
\begin{align}
\begin{split}
\mathop{\mathrm{Im}}\bigl(A\Theta(H-\kappa^2)\bigr)
&=
\tfrac12\bigl(f^{-1}\Theta-\Theta'\bigr) L
-2\lambda\mu\mathop{\mathrm{Re}}(\Theta A)
+(\lambda^2-\mu^2)\Theta' 
+\tfrac18 \Theta'''
\\&\phantom{{}={}}
-\bigl(q_1+\tfrac12q_2\bigr)\Theta'
-\tfrac12(\partial_fq_1)\Theta 
-\mathop{\mathrm{Im}}(q_2\Theta A)
\\&\phantom{{}={}}
+\mathop{\mathrm{Re}}\bigl((r_1+r_2)\Theta A\bigr)
+\tfrac12\mathop{\mathrm{Re}}\bigl(\Theta'(H-\kappa^2)\bigr)
\end{split}
\label{26052316}
\\
\begin{split}
&=
\tfrac12A\Theta' A
+\tfrac12f^{-1}\Theta L
-2\lambda\mu\mathop{\mathrm{Re}}(\Theta A)
-\tfrac18 \Theta'''
\\&\phantom{{}={}}
+\tfrac12q_2\Theta'
-\tfrac12(\partial_fq_1)\Theta 
-\mathop{\mathrm{Im}}(q_2\Theta A)
\\&\phantom{{}={}}
+\mathop{\mathrm{Re}}\bigl((r_1+r_2)\Theta A\bigr)
-\tfrac12\mathop{\mathrm{Re}}\bigl(\Theta'(H-\kappa^2)\bigr)
.
\end{split}
\label{26052315}
\end{align}
\end{lemma}

\begin{remark}
The representations \eqref{26052222} and \eqref{26052316} play central rolls 
in the proofs of the Agmon estimates, 
while \eqref{26052315} in that of Rellich's theorem. 
For appropriate weights $\Theta$ their right-hand sides 
are \textit{weightedly} positive or negative definite on eigenstates 
with negligible boundary terms. 
On the other hand, the identity \eqref{26052223} is used to control the 
non-self-adjoint contribution of the form 
$\mathop{\mathrm{Re}}(\Theta' A)\sim 4\lambda\mu\Theta$. 
Note that a first order differential operator is in general not bounded above or below, 
but on eigenstates $\mathop{\mathrm{Re}}(\Theta' A)$ is almost an explicit scalar. 
\end{remark}

\begin{proof}
By \eqref{260504} we can compute the first two quantities \eqref{26052222} and \eqref{26052223} as 
\[
\begin{split}
\mathop{\mathrm{Re}}(\Theta(H-\kappa^2))
&=
\tfrac12\mathop{\mathrm{Re}}(\Theta A^2)
+\tfrac12\mathop{\mathrm{Re}}(\Theta L)
+\mathop{\mathrm{Re}}(\Theta q)
-\mathop{\mathrm{Re}}(\Theta \kappa^2)
\\&=
\tfrac12A\Theta A
-\tfrac14\Theta''
+\tfrac12\Theta L
+(q_1+q_2)\Theta q
-(\lambda^2-\mu^2)\Theta
,
\end{split}
\]
and 
\[
\begin{split}
\mathop{\mathrm{Im}}(\Theta(H-\kappa^2))
&=
\tfrac12\mathop{\mathrm{Im}}(\Theta A^2)
+\tfrac12\mathop{\mathrm{Im}}(\Theta L)
+\mathop{\mathrm{Im}}(\Theta q)
-\mathop{\mathrm{Im}}(\Theta \kappa^2)
\\&=
\tfrac12\mathop{\mathrm{Re}}(\Theta' A)
+(r_1+r_2)\Theta 
-2\lambda\mu\Theta
.
\end{split}
\]
As for the third quantity, 
we again use \eqref{260504} to compute it as 
\begin{equation}
\begin{split}
\mathop{\mathrm{Im}}\bigl(A\Theta(H-\kappa^2)\bigr)
&=
\tfrac12\mathop{\mathrm{Im}}(A\Theta A^2)
+\tfrac12\mathop{\mathrm{Im}}(A\Theta L)
+\mathop{\mathrm{Im}}(A\Theta q)
-\mathop{\mathrm{Im}}(A\Theta \kappa^2)
\\&
=
\tfrac14A\Theta' A
+\tfrac12\mathop{\mathrm{Im}}(A\Theta L)
-\tfrac12\Theta' q_1
-\tfrac12\Theta (\partial_fq_1)
\\&\phantom{{}={}}
+\mathop{\mathrm{Im}}(A\Theta q_2)
+\mathop{\mathrm{Re}}\bigl(A\Theta (r_1+r_2)\bigr)
+\tfrac12(\lambda^2-\mu^2)\Theta'
\\&\phantom{{}={}}
-2\lambda\mu\mathop{\mathrm{Re}}(A\Theta )
.
\end{split}
\label{26050417}
\end{equation}
Here the second term on the right-hand side of \eqref{26050417} is further computed as 
\[
\begin{split}
\tfrac12\mathop{\mathrm{Im}}(A\Theta L)
&=
\tfrac14\mathop{\mathrm{Im}}\bigl(((\partial_jf)p_j+p_j(\partial_jf))\Theta p_k\ell_{kl}p_l\bigr)
\\&
=
\tfrac12\mathop{\mathrm{Im}}\bigl(p_kA\Theta \ell_{kl}p_l\bigr)
+\tfrac14\mathop{\mathrm{Re}}\bigl(((\partial_j\partial_kf)p_j+p_j(\partial_j\partial_kf))\Theta \ell_{kl}p_l\bigr)
\\&
=
-\tfrac14 \Theta'L
+\tfrac12f^{-1}\Theta L
,
\end{split}
\]
and thus by substituting the right above identity to \eqref{26050417} we obtain 
\begin{equation}
\begin{split}
\mathop{\mathrm{Im}}\bigl(A\Theta(H-\kappa^2)\bigr)
&=
\tfrac14A\Theta' A
+\bigl(\tfrac12f^{-1}\Theta-\tfrac14 \Theta'\bigr) L
+\tfrac12(\lambda^2-\mu^2)\Theta'
\\&\phantom{{}={}}
-2\lambda\mu\mathop{\mathrm{Re}}(\Theta A)
-\tfrac12 q_1\Theta'
-\tfrac12(\partial_fq_1)\Theta 
\\&\phantom{{}={}}
-\mathop{\mathrm{Im}}(q_2\Theta A)
+\mathop{\mathrm{Re}}\bigl((r_1+r_2)\Theta A\bigr)
.
\end{split}
\label{260527}
\end{equation}
Finally by \eqref{26052222} with $\Theta$ replace by $\Theta'$ we have 
\[
\begin{split}
0&=
\tfrac14A\Theta' A
+\tfrac14\Theta' L
-\tfrac12(\lambda^2-\mu^2)\Theta' 
-\tfrac18 \Theta'''
+\tfrac12(q_1+q_2)\Theta'
-\tfrac12\mathop{\mathrm{Re}}\bigl(\Theta'(H-\kappa^2)\bigr)
,
\end{split}
\]
and by subtracting and adding the right above identity from and to 
\eqref{260527}, we obtain the last expressions \eqref{26052316} and \eqref{26052315}, respectively. 
\end{proof}

\subsection{Refined Yosida approximation}

Next we introduce a refinement of the Yosida approximation as follows.

\begin{definition}\label{25071914480}
Let $\rho\in (0,1)$ be from Assumption~\ref{cond:20022317}, 
and define the \textit{refined Yosida approximation} 
with parameter $\nu\in\mathbb N$ as 
\begin{equation*}
\theta=\int_0^f(1+\tau/2^\nu)^{-1-\rho}\,\mathrm d\tau
=\rho^{-1}2^\nu \bigl(1-(1+f/2^\nu)^{-\rho}\bigr). 
\end{equation*}
\end{definition}

The above $\theta$ is a non-decreasing bounded function for each $\nu\in\mathbb N$, 
but monotonically converges to $f$ pointwise as $\nu\to\infty$. 
It also gets some additional decay properties as differentiated. 
Let us state them in the following lemma. 
Denote the derivatives of $\theta$ in $f$ by primes as before, and set for short 
\[
\theta_0=1+f/2^\nu
.
\]
Then, for example, we can write 
\begin{align}
\begin{split}
\theta'=\theta_0^{-1-\rho},\quad
\theta''=-(1+\rho) 2^{-\nu}\theta_0^{-2-\rho},
\quad \ldots.
\end{split}
\label{eq:12.5.1.19.56}
\end{align} 

\begin{lemma}\label{2605242351}
The refined Yosida approximation $\theta$ satisfies 
\begin{equation*}
0< 
\theta
%\int_0^f(1+\tau/2^{\nu})^{-1-\rho}\,\mathrm d\tau
\le \min\{f,\rho^{-1} 2^{\nu}\}\ \ \text{for any }\nu\in\mathbb N
,\qquad 
\theta
%\int_0^f(1+\tau/2^{\nu})^{-1-\rho}\,\mathrm d\tau
\uparrow f\ \ \text{pointwise as }\nu\to\infty
.
\end{equation*}
In addition, for any $k=3,4,\ldots$ there exists $C_k>0$ such that  
\begin{equation*}
%0\le (\theta-2s f)'=2\sigma\theta_0^{-1-\rho}
%,\qquad 
0< -\theta''\leq (1+\rho)f^{-1}\theta'
,\qquad 
0< (-1)^{k-1}\theta^{(k)}%\leq C_{k}f^{2-k}|\theta''|
\leq C_{k}f^{1-k}\theta'
%;\quad k=2,3,\dots.
.
\end{equation*}
\end{lemma}
\begin{proof}
The former properties are obvious by definition. 
The latter follows by \eqref{eq:12.5.1.19.56}
and $2^{-\nu}\theta_0^{-1}\le f^{-1}$. Thus we are done. 
\end{proof}

\section{Agmon estimates}

In this section we prove the Agmon estimates, or Theorem~\ref{260603}. 
We have to verify the critical exponential decay estimate for an eigenfunction from the critical 
exponential growth assumption, and there indeed seems to be a big gap. 
Let us divide it into the following three steps. 

\begin{proposition}\label{26060301a}
Suppose Assumption~\ref{cond:20022317} with (a), 
fix any $\kappa=\lambda+\mathrm i\mu\in \mathbb C_+\setminus\mathbb R$
and $s\in (0,\sqrt2\mu)$, 
and let $t>0$ be sufficiently large.
\begin{enumerate}
\item\label{26060301}
If $\phi\in  \mathcal E_{\sqrt2\mu,-t}$ solves \eqref{200125}, then $\phi\in \mathcal E_{s,-t}$. 
\item\label{26060302}
If $\phi\in  \mathcal E_{\mu-0}$ solves \eqref{200125}, then $\phi\in \mathcal E_{-\mu+0}$. 
\item\label{26060303}
If $\phi\in  \mathcal E_{-s,t}$ solves \eqref{200125}, then $\phi\in \mathcal E_{-\sqrt2\mu,t}$. 
\end{enumerate}
\end{proposition}

\begin{proof}[Deduction of Theorem~\ref{260603} from Proposition~\ref{26060301a}]
It is trivial. 
\end{proof}

Now the proof of Theorem~\ref{260603} is reduced to that of Proposition~\ref{26060301a}.
We will prove the assertions \ref{26060301}, \ref{26060302} and \ref{26060303} 
of Proposition~\ref{26060301a}, 
respectively, in Sections~\ref{2606035}, \ref{2606036} and \ref{2606037}. 
For their proofs we will repeat mutually look-alike commutator arguments, 
however, with different weight functions, 
showing positivity of \eqref{26052316}, \eqref{26052222}, and negativity of \eqref{26052316}, 
respectively. 
Each weight function contains as much as four or five parameters, 
but uniformity only in two of them is essential. 

The argument in Section~\ref{2606035} will be a prototype of those in 
Sections~\ref{2606036} and \ref{2606037}, and even in Section~\ref{26060716}. 
Detailed explanations may be omitted later than Section~\ref{2606035}.

\subsection{Small exponential growth estimates}\label{2606035}

First we prove the assertion \ref{26060301} of Proposition~\ref{26060301a}. 
A key is to focus on the representation \eqref{26052316}, not on \eqref{26052315}
as in our previous work~\cite{MR4062329} on Rellich's theorem. 
Take a weight function $\Theta$ of the form 
\begin{equation}
\Theta=\chi_{m,n}\mathrm e^{2\widetilde\theta}
;\quad 
\widetilde\theta=-\sqrt2\mu f+\bigl(\sqrt2\mu-s\bigr)\theta+tf^{1-\rho},
\label{2507191448e}
\end{equation}
depending on parameters $m,n,\nu\in\mathbb N$, $s\in (0,\sqrt2\mu)$ and 
$t\ge 1$ with $n>m$. 
Recall that $\theta$ is the refined Yosida approximation from Definition~\ref{25071914480},
and $\chi_{m,n}$ is a smooth cut-off function defined as 
\begin{equation}
\chi_m=\chi(f/2^m)
,\quad 
\bar\chi_m=1-\chi_m
,\quad 
\chi_{m,n}=\bar\chi_m\chi_n
,
\label{eq:11.7.11.5.14}
\end{equation}
where $\chi\in C^\infty(\mathbb R)$ obeys \eqref{eq:14.1.7.23.24}. 

\begin{lemma}\label{26060317e}
Suppose Assumption~\ref{cond:20022317} with (a), 
and fix any $\kappa=\lambda+\mathrm i\mu\in\mathbb C_+\setminus\mathbb R$
and $s\in (0,\sqrt2\mu)$. 
Then there exist $c,C> 0$, $m\in\mathbb N$ and $t\ge 1$ such that 
uniformly in $n>m$ and $\nu\in\mathbb N$
\[
\begin{split}
\mathop{\mathrm{Im}}\bigl(A\Theta(H-\kappa^2)\bigr)
&\ge c\Theta-C\bigl(\chi_{m-1,m+1}^2+\chi_{n-1,n+1}^2\bigr)f^{-1}\mathrm e^{2\widetilde\theta}
\\&\phantom{{}={}}
-C(H-\kappa^2)^*\chi_{m-1,n+1}f^{\rho}\mathrm e^{2\widetilde\theta}(H-\kappa^2)
.
\end{split}
\]  
\end{lemma}
\begin{remarks}
\begin{enumerate}
\item
The above inequality might seem extremely technical, but we may consider it
a variant of the divergence theorem. 
Recall that the left-hand and the last term on the right-hand side vanish
on an eigenstate $\phi$. 
Then it only suggests that the internal integral of $\Theta|\phi|^2$ in a region 
$\{2^m\lesssim f\lesssim 2^n\}$ 
is controlled by \textit{thick} boundary terms around $\{f\sim 2^m\}$ and $\{f\sim 2^n\}$. 
\item
Lemma~\ref{26060317e} verifies the assertion \ref{26060301} of Proposition~\ref{26060301a} as follows.
Take an expectation value of the inequality on $\phi$, and let $n\to\infty$. 
Then the outer boundary term vanishes due to the assumption,
and the integral of $\mathrm e^{2\widetilde\theta}|\phi|^2$ on $\{f\gtrsim 2^m\}$ is bounded by 
the inner boundary term uniformly in $\nu$. 
Next, let $\nu\to\infty$, and an additional exponential decay follows for $\phi$. 
\end{enumerate}
\end{remarks}

\begin{proof}
Let $\kappa=\lambda+\mathrm i\mu\in\mathbb C_+\setminus\mathbb R$ and 
$s\in (0,\sqrt2\mu)$ be fixed. 
On the other parameters $n,m,\nu\in\mathbb N$ and $t\ge 1$,
we first assume only $n>m$,
and an additional restriction will be imposed when it appears necessary.
The below estimates are uniform in these parameters 
under the restrictions until then,
and a constant $C>0$ appearing below, possibly retaken larger line by line, 
is independent of them. 

Let us start to compute the left-hand side of the asserted inequality. 
In the following, for simplicity, we gather \textit{admissible error terms}, and absorb them into 
\begin{equation}
\begin{split}
Q&=
f^{-\rho}\Theta
+t^2f^{-2\rho}\Theta
+t^2\bigl(|\chi_{m,n}'|+|\chi_{m,n}''|+|\chi_{m,n}'''|\bigr)\mathrm e^{2\widetilde\theta}
\\&\phantom{{}={}}
+p_j\bigl(\chi_{m,n}f^{-\rho}+|\chi_{m,n}'|\bigr)\mathrm e^{2\widetilde\theta} p_j
+(H-\kappa^2)^*\bigl(f^{\rho}\Theta+|\chi_{m,n}'|\mathrm e^{2\widetilde\theta}\bigr)(H-\kappa^2)
,
\end{split}
\label{260524e}
\end{equation}
which is dependent on $t$. 
It is very useful, e.g., since, once a differentiation acts on $\chi_{m,n}$, 
then the corresponding term is absorbed into $Q$. 
We will estimate it later on. We also note that in the following arguments 
we shall freely use \eqref{2507191448e}, \eqref{260524e}, \eqref{eq:12.5.1.19.56}, Lemma~\ref{2605242351}
and the Cauchy--Schwarz inequality without reference. 
Now by using the expression \eqref{26052316} we have 
\begin{equation}
\begin{split}
\mathop{\mathrm{Im}}\bigl(A\Theta(H-\kappa^2)\bigr)
&\ge 
\tfrac12\bigl(f^{-1}\Theta-\widetilde\theta'\Theta\bigr) L
-2\lambda\mu\mathop{\mathrm{Re}}(\Theta A)
\\&\phantom{{}={}}
+2\bigl(\lambda^2-\mu^2\bigr)\widetilde\theta'\Theta
+\widetilde\theta'^3\Theta
-CQ
.
\end{split}
\label{26052420e}
\end{equation}
Let us take a closer look at the second term on the right-hand side of \eqref{26052420e}. 
Let $m$ be large enough, depending on $t$, so that 
\begin{equation}
\widetilde\theta'\le -C^{-1}\ \ \text{on }\mathop{\mathrm{supp}}\chi_{m,n}
.
\label{26061116e}
\end{equation}
Then we can rewrite and estimate the second term of \eqref{26052420e}
by using \eqref{26052223} as
\begin{equation}
\begin{split}
-2\lambda\mu\mathop{\mathrm{Re}}(\Theta A)
&\ge 
\lambda\mu\bigl(-\widetilde\theta'\bigr)^{-1/2}\mathop{\mathrm{Re}}(\Theta' A)\bigl(-\widetilde\theta'\bigr)^{-1/2}
-CQ
\\&\ge 
-4\lambda^2\mu^2\widetilde\theta'\Theta
-CQ
.
\end{split}
\label{260611e}
\end{equation}
Thus by combining \eqref{26052420e} and \eqref{260611e}
and then using \eqref{26061116e} we obtain 
\begin{equation}
\begin{split}
\mathop{\mathrm{Im}}\bigl(A\Theta(H-\kappa^2)\bigr)
&\ge 
\tfrac12\bigl(f^{-1}\Theta-\widetilde\theta'\Theta\bigr) L
-\widetilde\theta'^{-1}\bigl(2\lambda^2+\widetilde\theta'^2\bigr)\bigl(2\mu^2-\widetilde\theta'^2\bigr)\Theta 
\\&\phantom{{}={}}
-CQ
\\&\ge 
C^{-1}tf^{-\rho}\Theta 
-CQ
.
\end{split}
\label{26061115e}
\end{equation}
Finally we estimate $Q$. Since the second term of \eqref{260524e} can be 
rewritten as 
\[
\begin{split}
p_j\bigl(\chi_{m,n}f^{-\rho}+|\chi_{m,n}'|\bigr)\mathrm e^{2\widetilde\theta} p_j
&=
\tfrac12\bigl(f^{-\rho}\Theta+|\chi_{m,n}'|\mathrm e^{2\widetilde\theta}\bigr)''
\\&\phantom{{}={}}
+2\bigl(f^{-\rho}\Theta+|\chi_{m,n}'|\mathrm e^{2\widetilde\theta}\bigr) \bigl(\lambda^2-\mu^2-q_1-q_2\bigr)
\\&\phantom{{}={}}
+2\mathop{\mathrm{Re}}\bigl((f^{-\rho}\Theta+|\chi_{m,n}'|\mathrm e^{2\widetilde\theta}) (H-\kappa^2)\bigr)
,
\end{split}
\]
we have 
\begin{equation}
\begin{split}
Q&
\le 
Cf^{-\rho}\Theta
+Ct^2f^{-2\rho}\Theta
+Ct^2\bigl(\chi_{m-1,m+1}^2+\chi_{n-1,n+1}^2\bigr)f^{-1}\mathrm e^{2\widetilde\theta}
\\&\phantom{{}={}}
+C(H-\kappa^2)^*\chi_{m-1,n+1}f^{\rho}\mathrm e^{2\widetilde\theta}(H-\kappa^2)
.
\end{split}
\label{26060922e}
\end{equation}
Thus by substituting \eqref{26060922e} into \eqref{26061115e} we obtain 
\begin{equation}
\begin{split}
\mathop{\mathrm{Im}}\bigl(A\Theta(H-\kappa^2)\bigr)
&\ge 
\bigl(C^{-1}t-C-Ctf^{-\rho}\bigr)f^{-\rho}\Theta
\\&\phantom{{}={}}
-Ct^2\bigl(\chi_{m-1,m+1}^2+\chi_{n-1,n+1}^2\bigr)f^{-1}\mathrm e^{2\widetilde\theta}
\\&\phantom{{}={}}
-C(H-\kappa^2)^*\chi_{m-1,n+1}f^{\rho}\mathrm e^{2\widetilde\theta}(H-\kappa^2)
.
\end{split}
\label{26061118}
\end{equation}
Now, fix $t$ large enough, 
which also makes $m$ large to keep \eqref{26061116e}, 
but then retake $m$ even larger if necessary 
to make the coefficient of the first term on the above right-hand side positive. 
Then this in fact verifies the assertion. 
\end{proof}

\begin{proof}[Proof of the assertion \ref{26060301} of Proposition~\ref{26060301a}]
Fix any $\kappa=\lambda+\mathrm i\mu\in \mathbb C_+\setminus\mathbb R$ 
and $s\in (0,\sqrt2\mu)$, and take $m\in\mathbb N$ and $t\ge 1$ 
as in Lemma~\ref{26060317e}.
Let $\phi\in \mathcal E_{\sqrt2\mu,-t}$ be a solution to \eqref{200125}. 
Then take an expectation value of the inequality from Lemma~\ref{26060317e}
on the state $\chi_{m-2,n+2}\phi$,
and it follows that for any $n>m$ and $\nu\in\mathbb N$
\begin{equation}
\begin{split}
\bigl\|\chi_{m,n}\mathrm e^{-\sqrt2\mu f+(\sqrt2\mu-s)\theta+tf^{1-\rho}}\phi\bigr\|_{L^2}^2
&\le 
C\|\chi_{m-1,m+1}\phi\|_{L^2}^2
\\&\phantom{{}={}}
+C_\nu\bigl\|\chi_{n-1,n+1}f^{-1/2}\mathrm e^{-\sqrt2\mu f+tf^{1-\rho}}\phi\bigr\|_{L^2}^2,
\end{split}\label{2606031731e}
\end{equation}
where $C_\nu$ depends on $\nu$, but not on $n$, see Lemma~\ref{2605242351}. 
Let $n\to\infty$, and then the second term on the right-hand side of \eqref{2606031731e}
vanishes due to the assumption $\phi\in \mathcal E_{\sqrt2\mu,-t}$. 
On the other hand, we also apply Lebesgue's monotone convergence theorem to the left-hand side of \eqref{2606031731e}, 
and thus it follows that 
\begin{equation*}
\bigl\|\bar\chi_{m}\mathrm e^{-\sqrt2\mu f+(\sqrt2\mu-s)\theta+tf^{1-\rho}}\phi\bigr\|_{L^2}^2
\le 
C\|\chi_{m-1,m+1}\phi\|_{L^2}^2
.
\end{equation*}
Next we let $\nu\to\infty$, and by Lebesgue's monotone convergence theorem again 
we obtain 
\[
\bigl\|\bar\chi_{m}\mathrm e^{-sf+tf^{1-\rho}}\phi\bigr\|_{L^2}^2
\le 
C\|\chi_{m-1,m+1}\phi\|_{L^2}^2
.
\]
This implies the assertion. We are done. 
\end{proof}

\subsection{Small exponential decay estimates}\label{2606036}

We next prove the assertion \ref{26060302} of Proposition~\ref{26060301a}. 
This time we focus on the quantity 
\eqref{26052222}, not on \eqref{26052316} or \eqref{26052315}. 
Let 
\begin{equation}
\Theta=\chi_{m,n}\mathrm e^{2\widetilde \theta}
;\quad 
\widetilde\theta=-s f+(s+t)\theta, 
\label{2507191448f}
\end{equation}
with parameters $m,n,\nu\in\mathbb N$ and  $s,t\in (0,\mu)$ with $n>m$. 
Here $\chi_{m,n}$ and $\theta$ are from \eqref{eq:11.7.11.5.14} and Definition~\ref{25071914480},
respectively.

\begin{lemma}\label{26060317f}
Suppose Assumption~\ref{cond:20022317} with (a), 
and fix any $\kappa=\lambda+\mathrm i\mu\in\mathbb C_+\setminus\mathbb R$ and $s,t\in(0,\mu)$. 
Then there exist $c,C> 0$ and $m\in\mathbb N$ such that 
uniformly in $n>m$ and $\nu\in\mathbb N$
\[
\begin{split}
\mathop{\mathrm{Re}}\bigl(\Theta(H-\kappa^2)\bigr)
&\ge c\Theta
-C\bigl(\chi_{m-1,m+1}^2+\chi_{n-1,n+1}^2\bigr)f^{-1}\mathrm e^{2\widetilde \theta}
\\&\phantom{{}={}}
-C(H-\kappa^2)^*\chi_{m-1,n+1}f^{\rho}\mathrm e^{2\widetilde \theta}(H-\kappa^2)
.
\end{split}
\]  
\end{lemma}
\begin{proof}
%We basically argue similarly to the proof of Lemma~\ref{26060317e}. 
Fix any $\kappa=\lambda+\mathrm i\mu\in\mathbb C_+\setminus\mathbb R$ and $s,t\in(0,\mu)$. 
For the moment the below estimates are uniform in $n>m\ge 1$ and $\nu\in\mathbb N$,
and constants $C>0$ are independent of them. 
Here the admissible error terms are of the form 
\begin{equation*}
\begin{split}
Q&=
f^{-\rho}\Theta
+|\chi_{m,n}'|\mathrm e^{2\widetilde \theta}
+p_j\bigl(f^{-\rho}\Theta+|\chi_{m,n}'|\mathrm e^{2\widetilde \theta}\bigr) p_j
+(H-\kappa^2)^*f^{\rho}\Theta(H-\kappa^2)
.
\end{split}
%\label{260524f}
\end{equation*}
%and we will use \eqref{2507191448f}, \eqref{260524f}, Lemma~\ref{2605242351}
%and the Cauchy--Schwarz inequality without mentioning it. 
Then by \eqref{26052222} and \eqref{26052223} we can compute and estimate the left-hand side of the asserted inequality as 
\begin{equation*}
\begin{split}
\mathop{\mathrm{Re}}\bigl(\Theta(H-\kappa^2)\bigr)
&\ge 
\tfrac12\bigl(A-2\lambda\mu^{-1}\theta'\bigr)\Theta \bigl(A-2\lambda\mu^{-1}\theta'\bigr)
+2\lambda\mu^{-1}\mathop{\mathrm{Re}}(\theta'\Theta A)
\\&\phantom{{}={}}
-2\lambda^2\mu^{-2}\theta'^2\Theta 
+\tfrac12\Theta L
-(\lambda^2-\mu^2)\Theta 
-\theta'^2\Theta
-CQ
\\&\ge 
\lambda\mu^{-1}\mathop{\mathrm{Re}}(\Theta' A)
-\lambda^2\bigl(1+2\mu^{-2}\theta'^2\bigr)\Theta 
+(\mu^2-\theta'^2)\Theta
-CQ
\\&\ge 
\bigl(\lambda^2(3-2\mu^{-2}\theta'^2)+(\mu^2-\theta'^2)\bigr)\Theta
-CQ
\\&\ge 
C^{-1}\Theta
-CQ
.
\end{split}
%\label{26060416}
\end{equation*}
Finally, if we estimate $Q$ similarly to \eqref{26060922e}, we obtain 
\begin{equation*}
\begin{split}
\mathop{\mathrm{Im}}\bigl(A\Theta(H-\kappa^2)\bigr)
&\ge 
\bigl(C^{-1}-Cf^{-\rho}\bigr)\Theta
-C\bigl(\chi_{m-1,m+1}^2+\chi_{n-1,n+1}^2\bigr)f^{-1}\mathrm e^{2\widetilde\theta}
\\&\phantom{{}={}}
-C(H-\kappa^2)^*\chi_{m-1,n+1}f^{\rho}\mathrm e^{2\widetilde\theta}(H-\kappa^2)
.
\end{split}
\end{equation*}
Hence by letting $m\in\mathbb N$ be large enough we obtain the assertion. 
\end{proof}

\begin{proof}[Proof of the assertion \ref{26060302} of Proposition~\ref{26060301a}]
Let $\kappa=\lambda+\mathrm i\mu\in \mathbb C_+\setminus\mathbb R$, 
and assume that $\phi\in  \mathcal E_{\mu-0}$ solves \eqref{200125}. 
Then we can find $s\in (0,\mu)$ such that $\phi\in\mathcal E_{s}$. 
Take any $t\in (0,\mu)$, and choose $m\in\mathbb N$ as in Lemma~\ref{26060317f}. 
Then by using Lemma~\ref{26060317f} 
we can discuss similarly to the proof of the assertion \ref{26060301} of the same proposition, 
and finally obtain 
%it follows that uniformly in $n>m$ and $\nu\in\mathbb N$
%\begin{equation*}
%\bigl\|\chi_{m,n}\mathrm e^{s\theta-\mu f}\phi\bigr\|_{L^2}^2
%\le 
%C\bigl\|\chi_{m-1,m+1}\phi\bigr\|_{L^2}^2
%+C_\nu\bigl\|\chi_{n-1,n+1}f^{-1/2}\mathrm e^{-\mu f}\phi\bigr\|_{L^2}^2
%,
%\label{2606031731f}
%\end{equation*}
%where $C_\nu$ depends on $\nu$, but not on $n$. 
%Let $n\to\infty$, and then by the assumption and 
%Lebesgue's monotone convergence theorem 
%\begin{equation*}
%\bigl\|\bar\chi_{m}\mathrm e^{s\theta-\mu f}\phi\bigr\|_{L^2}^2
%\le 
%C\bigl\|\chi_{m-1,m+1}\phi\bigr\|_{L^2}^2
%.
%\end{equation*}
%Next let $\nu\to\infty$, and then by Lebesgue's monotone convergence theorem again 
\[
\bigl\|\bar\chi_{m}\mathrm e^{tf}\phi\bigr\|_{L^2}^2
\le 
C\|\chi_{m-1,m+1}\phi\|_{L^2}^2
.
\]
Hence we are done. 
\end{proof}

\subsection{Critical exponential decay estimates}\label{2606037}

Finally in this section we prove the assertion \ref{26060303} 
of Proposition~\ref{26060301a}. 
Here we verify negativity of quantity \eqref{26052316}, 
see also Tagawa~\cite{MR4926373}. 
We let 
\begin{equation}
\Theta=\chi_{m,n}\mathrm e^{2\widetilde\theta}
;\quad 
\widetilde\theta=sf+\bigl(\sqrt2\mu-s)\theta-tf^{1-\rho}
\label{2507191448h}
\end{equation}
depending on parameters $m,n,\nu\in\mathbb N$, $s\in (0,\sqrt2\mu)$ and $t\ge 1$ with $n>m$. 
Recall that $\chi_{m,n}$ and $\theta$ are from \eqref{eq:11.7.11.5.14} and Definition~\ref{25071914480}, respectively.

\begin{lemma}\label{26060317h}
Suppose Assumption~\ref{cond:20022317} with (a), 
and fix any $\kappa=\lambda+\mathrm i\mu\in\mathbb C_+\setminus\mathbb R$ and 
$s\in (0,\sqrt2\mu)$. 
Then there exist $c,C> 0$, $m\in\mathbb N$ and $t\ge 1$ such that 
uniformly in $n>m$ and $\nu\in\mathbb N$
\[
\begin{split}
\mathop{\mathrm{Im}}\bigl(A\Theta(H-\kappa^2)\bigr)
&\le -c\Theta
+C\bigl(\chi_{m-1,m+1}^2+\chi_{n-1,n+1}^2\bigr)f^{-1}\mathrm e^{2\widetilde\theta}
\\&\phantom{{}={}}
+C(H-\kappa^2)^*\chi_{m-1,n+1}f^{\rho}\mathrm e^{2\widetilde\theta}(H-\kappa^2)
.
\end{split}
\]  
\end{lemma}
\begin{proof}
Fix any $\kappa=\lambda+\mathrm i\mu\in\mathbb C_+\setminus\mathbb R$ 
and $s\in (0,\sqrt2\mu)$. 
Except for $n>m$ we for the moment do not restrict parameters 
$n,m,\nu\in\mathbb N$ and $t\ge 1$, 
and constants $C>0$ are independent of them. 
The admissible error terms here are of the form
\begin{equation*}
\begin{split}
Q&=
f^{-\rho}\Theta
+t^2f^{-2\rho}\Theta
+t^2\bigl(\chi_{m,n}|\chi_{m,n}'|+|\chi_{m,n}''|+|\chi_{m,n}'''|\bigr)\mathrm e^{2\widetilde\theta}
\\&\phantom{{}={}}
+p_j\bigl(f^{-\rho}\Theta+|\chi_{m,n}'|\mathrm e^{2\widetilde\theta}\bigr) p_j
+(H-\kappa^2)^*\bigl(f^{\rho}\Theta+|\chi_{m,n}'|\mathrm e^{2\widetilde\theta}\bigr)(H-\kappa^2)
.
\end{split}
%\label{260524h}
\end{equation*}
Note that it has the same form as \eqref{260524e}, 
and we have the same estimate \eqref{26060922e} for it although $\widetilde\theta$ 
is slightly different. 
%In the following arguments we will use \eqref{2507191448h}, \eqref{260524h}, Lemma~\ref{2605242351}
%and the Cauchy--Schwarz inequality without mentioned. 
Now, similarly to \eqref{26052420e}, we have by \eqref{26052316} 
\begin{equation}
\begin{split}
\mathop{\mathrm{Im}}\bigl(A\Theta(H-\kappa^2)\bigr)
&\le 
\bigl(\tfrac12f^{-1}\Theta-\widetilde\theta'\Theta\bigr) L
-2\lambda\mu\mathop{\mathrm{Re}}(\Theta A)
\\&\phantom{{}={}}
+2\bigl(\lambda^2-\mu^2\bigr)\widetilde\theta'\Theta
+\widetilde\theta'^3\Theta
+CQ
.
\end{split}
\label{260529}
\end{equation}
The second term of \eqref{260529} is estimated similarly to \eqref{260611e}. 
Here we take $m$ large enough, depending on $t$, such that 
\begin{equation}
\widetilde\theta'\ge C^{-1}\ \ \text{on }\mathop{\mathrm{supp}}\chi_{m,n}
,
\label{26061116}
\end{equation}
and then obtain 
\begin{equation}
\begin{split}
-2\lambda\mu\mathop{\mathrm{Re}}(\Theta A)
&\le 
-4\lambda^2\mu^2\widetilde\theta'^{-1}\Theta
+CQ
,
\end{split}
\label{260611}
\end{equation} 
see \eqref{260611e}. 
Thus by \eqref{260529} and \eqref{260611} along with 
\eqref{26061116} and \eqref{26060922e} it follows that 
\begin{equation*}
\begin{split}
\mathop{\mathrm{Im}}\bigl(A\Theta(H-\kappa^2)\bigr)
&\le 
\bigl(\tfrac12f^{-1}\Theta-\widetilde\theta'\Theta\bigr) L
-\bigl(C^{-1}t-C-Ctf^{-\rho}\bigr)f^{-\rho}\Theta 
\\&\phantom{{}={}}
+Ct^2\bigl(\chi_{m-1,m+1}^2+\chi_{n-1,n+1}^2\bigr)f^{-1}\mathrm e^{2\widetilde\theta}
\\&\phantom{{}={}}
+C(H-\kappa^2)^*\chi_{m-1,n+1}f^{\rho}\mathrm e^{2\widetilde\theta}(H-\kappa^2)
,
\end{split}
%\label{260529}
\end{equation*}
cf.\ \eqref{26061118}. 
Thus by letting $t$ large enough and then $m$ larger if necessary. 
we obtain the assertion.
\end{proof}

\begin{proof}[Proof of the assertion \ref{26060303} of Proposition~\ref{26060301a}]
Fix $\kappa=\lambda+\mathrm i\mu\in \mathbb C_+\setminus\mathbb R$ and 
$s\in (0,\sqrt2\mu)$, and choose $m\in\mathbb N$ and 
$t\ge 1$ as in Lemma~\ref{26060317h}. 
Assume that $\phi\in \mathcal E_{-s,t}$ solves \eqref{200125}.
We can find $t\in (0,\sqrt2\mu)$ such that $\phi\in (\mathcal B^*_0)_{t,-1/2}$. 
Then by using Lemma~\ref{26060317h} 
we can prove, similarly to the proof of the assertions \ref{26060301} of 
the same proposition, 
%for any $n>m$ and $\nu\in\mathbb N$ we have 
%\begin{equation*}
%\bigl\|\chi_{m,n}\mathrm e^{tf+s\theta}\phi\bigr\|_{L^2}^2
%\le 
%C\bigl\|\chi_{m-1,m+1}\phi\bigr\|_{L^2}^2
%+C_\nu\bigl\|\chi_{n-1,n+1}f^{-1/2}\mathrm e^{t f}\phi\bigr\|_{L^2}^2
%\label{2606031731i}
%\end{equation*}
%with $C_\nu$ dependent on $\nu$, but not on $n$. 
%Now let $n\to\infty$, and then $\nu\to\infty$, and we obtain 
\[
\bigl\|\bar\chi_{m}\mathrm e^{\sqrt2\mu f-tf^{1-\rho}}\phi\bigr\|_{L^2}^2
\le 
C\|\chi_{m-1,m+1}\phi\|_{L^2}^2
.
\]
Thus we are done. 
\end{proof}

\section{Rellich's theorem}\label{26060716}

In this section we prove Rellich's theorem, or Theorem~\ref{2606034}. 
We split the proof into several parts. 
Although their statements and proofs look very similar, we have to show them separately. 

For non-zero spectral parameters we first deduce additional exponential decay 
from the assumptions.

\begin{proposition}\label{20260535017a}
Suppose Assumption~\ref{cond:20022317} with (b) or (c), 
fix any $\kappa=\lambda+\mathrm i\mu\in \mathbb C_+\setminus \mathbb R$
or $\mathbb R\setminus \{0\}$, respectively, 
and let $t\ge 1$ be sufficiently large. 
Then, if $\phi\in \mathcal E_{-\sqrt2\mu,-t}$ solves \eqref{200125}, 
one has  $\phi\in\mathcal E_{-\sqrt2\mu-0}$. 

\end{proposition}

We then deduce super-exponential decay estimates, and verify conclusions of Theorem~\ref{2606034}
for $\kappa\in\mathbb C_+\setminus\{0\}$. We note that this process can be done only with 
Assumption~\ref{cond:20022317} (b) even for $\kappa\in \mathbb R$.  

\begin{proposition}\label{20260535017}
Suppose Assumption~\ref{cond:20022317} with (b),
and let $\kappa=\lambda+\mathrm i\mu\in \mathbb C_+$.
\begin{enumerate}
\item
If $\phi\in \mathcal E_{-\sqrt2\mu-0}$ solves \eqref{200125}, then $\phi\in \mathcal E_{-\infty}$. 
\item
If $\phi\in \mathcal E_{-\infty}$ solves \eqref{200125}, 
then $\phi=0$ on $\mathbb R^d$. 
\end{enumerate}
\end{proposition}

On the other hand, for zero spectral parameter we cannot expect any additional exponential decay,
and we go through a subexponential version. Here we can do it in two steps. 

\begin{proposition}\label{20260535017b}
Suppose Assumption~\ref{cond:20022317} with (d), 
let $\kappa=0$, and let $t\ge 1$ be sufficiently large.
\begin{enumerate}
\item
If $\phi\in \mathcal E_{0,-t}$ solves \eqref{200125}, then $\phi\in \mathcal E_{0,-\infty}$. 
\item
If $\phi\in \mathcal E_{0,-\infty}$ solves \eqref{200125}, 
then $\phi=0$ on $\mathbb R^d$. 
\end{enumerate}
\end{proposition}

\begin{remark}
The above assertion 2 is in close relation with the Landis conjecture. 
See the literature referred to in Remarks~\ref{260704}. 
\end{remark}

\begin{proof}[Deduction of Theorem~\ref{2606034} from Propositions~\ref{20260535017a}, \ref{20260535017}
and \ref{20260535017b}]
It is trivial. 
\end{proof}

We prove Propositions~\ref{20260535017a}, \ref{20260535017} and \ref{20260535017b}
in Sections~\ref{2606070}, \ref{2606071} and \ref{2606072}, respectively. 
We will further split the proofs according to forms of weight functions.  
All of these proofs are done by the commutator arguments, 
similarly to the Agmon estimates in the previous section, 
however this time we will only focus on positivity of the representation \eqref{26052315},
and dependence on parameters is a bit more involved.

Each of the following commutator arguments employs techniques similar to those in Section~\ref{2606035}, 
and we may omit some of the details. See Section~\ref{2606035} for them.

\subsection{Additional exponential decay estimates}\label{2606070}

\subsubsection{Non-real spectral parameters}

We start with Proposition~\ref{20260535017a} with $\kappa=\lambda+\mathrm i\mu\in\mathbb C_+\setminus\mathbb R$. 
The weight function here is  
\begin{equation}
\Theta=\chi_{m,n}\mathrm e^{2\widetilde\theta}
;\quad 
\widetilde\theta=\sqrt2\mu f+s\theta+tf^{1-\rho}
,
\label{2507191448d}
\end{equation}
with parameters $m,n,\nu\in\mathbb N$, $s\in (0,1]$ and $t\ge 1$ 
obeying $n>m$.
Recall that 
$\chi_{m,n}$ and $\theta$ are from \eqref{eq:11.7.11.5.14} and Definition~\ref{25071914480}, respectively. 

\begin{lemma}\label{260612}
Suppose Assumption~\ref{cond:20022317} with (b), and let $\kappa=\lambda+\mathrm i\mu\in\mathbb C_+\setminus\mathbb R$. 
Then there exist $c,C$, $m\in\mathbb N$, $s\in (0,1]$ and $t\ge 1$ 
such that uniformly in $n>m$ and $\nu\in\mathbb N$
\[
\begin{split}
\mathop{\mathrm{Im}}\bigl(A\Theta(H-\kappa^2)\bigr)
&\ge cf^{-1-\rho}\Theta-C\bigl(\chi_{m-1,m+1}^2+\chi_{n-1,n+1}^2\bigr)f^{-1}\mathrm e^{2\widetilde\theta}
\\&\phantom{{}={}}
-C(H-\kappa^2)^*\chi_{m-1,n+1}f^{1+\rho}\mathrm e^{2\widetilde\theta}(H-\kappa^2)
.
\end{split}
\]  
\end{lemma}
\begin{remark}
Technically, it has been very difficult to control the fourth term of \eqref{26052317l} below, 
since it produces a strong negative contribution if we apply \eqref{26052223} to it as it is. 
However, it turned out that we could squeeze a positive contribution from 
the third term of \eqref{26052317l} to compensate it. 
\end{remark}
\begin{proof}
Fix any $\kappa=\lambda+\mathrm i\mu\in\mathbb C_+\setminus\mathbb R$. 
Note that the below estimates are uniform in $n,m,\nu\in\mathbb N$, $s\in (0,1]$ and $t\ge 1$ with 
$n>m$, and constants $C>0$ are independent of them. 
In this proof the admissible error terms are of the form
\begin{equation}
\begin{split}
Q&=
f^{-1-\rho}\Theta
+t^2f^{-1-2\rho}\Theta
+t^2\bigl(|\chi_{m,n}'|+|\chi_{m,n}''|+|\chi_{m,n}'''|\bigr)\mathrm e^{2\widetilde\theta}
\\&\phantom{{}={}}
+p_j\bigl(f^{-1-\rho}\Theta+|\chi_{m,n}'|\mathrm e^{2\widetilde\theta} \bigr)p_j
\\&\phantom{{}={}}
+(H-\kappa^2)^*\bigl(f^{1+\rho}\Theta+|\chi_{m,n}'|\mathrm e^{2\widetilde\theta}\bigr)(H-\kappa^2)
,
\end{split}
\label{260524l}
\end{equation}
which depends on $t$. 
Then by the expression \eqref{26052315} we have 
\begin{equation}
\begin{split}
\mathop{\mathrm{Im}}\bigl(A\Theta(H-\kappa^2)\bigr)
&\ge 
A\widetilde\theta'\Theta A
+\tfrac12f^{-1}\Theta L
-2\lambda\mu\mathop{\mathrm{Re}}(\Theta A)
\\&\phantom{{}={}}
-\widetilde\theta'^3\Theta 
-\tfrac32\widetilde\theta'\widetilde\theta''\Theta 
+\mathop{\mathrm{Re}}(r_2\Theta A)
-CQ
\\&\ge 
\tfrac12Af^{-1}\Theta A
+\tfrac12f^{-1}\Theta L
+A\bigl(\widetilde\theta'-\tfrac12f^{-1}\bigr)\Theta A
\\&\phantom{{}={}}
-\mathop{\mathrm{Re}}\bigl((2\lambda\mu-r_2)\Theta A\bigr)
-\widetilde\theta'^3\Theta 
-\tfrac32\widetilde\theta'\widetilde\theta''\Theta 
-CQ
.
\end{split}
\label{26052317l}
\end{equation}
The first and the second terms on the right-hand side 
of \eqref{26052317l} can be combined by using \eqref{26052222} with 
$\Theta$ replaced by $f^{-1}\Theta$, 
and they produce a positive contribution as 
\begin{equation}
\begin{split}
\tfrac12Af^{-1}\Theta A
+\tfrac12f^{-1}\Theta L
&\ge 
f^{-1}\widetilde\theta'^2\Theta
+\bigl(\lambda^2-\mu^2\bigr)f^{-1}\Theta 
-CQ
.
\end{split}
\label{26052422l}
\end{equation}
Meanwhile, a strong negative contribution from the fifth and the sixth terms of \eqref{26052317l} 
and an indefinite one from the fourth
can be moderated with the third. 
In fact, if we set 
\[
B=A-\tfrac12(4\lambda\mu-r_2)\widetilde\theta'^{-1}-\mathrm i\widetilde\theta',
\]
then we can combine and rewrite the third to the sixth terms of \eqref{26052317l} as 
\begin{equation}
\begin{split}
&
A\bigl(\widetilde\theta'-\tfrac12f^{-1}\bigr)\Theta A
-\mathop{\mathrm{Re}}\bigl((2\lambda\mu-r_2)\Theta A\bigr)
-\widetilde\theta'^3\Theta 
-\tfrac32\widetilde\theta'\widetilde\theta''\Theta 
\\&
=
B^*\bigl(\widetilde\theta'-\tfrac12f^{-1}\bigr)\Theta B
+2\lambda\mu\mathop{\mathrm{Re}}\bigl(\widetilde\theta'^{-1}\bigl(\widetilde\theta'-f^{-1}\bigr)\Theta A\bigr)
+\tfrac12\mathop{\mathrm{Re}}\bigl(r_2f^{-1}\widetilde\theta'^{-1}\Theta A\bigr)
\\&\phantom{{}={}}
+\bigl(\widetilde\theta'(\widetilde\theta'-\tfrac12f^{-1})\Theta\bigr)' 
-\bigl(\tfrac14(4\lambda\mu-r_2)^2\widetilde\theta'^{-2}+\widetilde\theta'^2\bigr)\bigl(\widetilde\theta'-\tfrac12f^{-1}\bigr)\Theta 
\\&\phantom{{}={}}
-\widetilde\theta'^3\Theta 
-\tfrac32\widetilde\theta'\widetilde\theta''\Theta 
\\&
\ge 
B^*\bigl(\widetilde\theta'-\tfrac12f^{-1}\bigr)\Theta B
+\lambda\mu\widetilde\theta'^{-1/2}\mathop{\mathrm{Re}}(\Theta' A)\widetilde\theta'^{-1/2}
\\&\phantom{{}={}}
-\lambda\mu f^{-1/2}\widetilde\theta'^{-1}\mathop{\mathrm{Re}}(\Theta' A)f^{-1/2}\widetilde\theta'^{-1}
-4\lambda^2\mu^2\widetilde\theta'^{-1}\Theta 
+2\lambda^2\mu^2f^{-1}\widetilde\theta'^{-2}\Theta
\\&\phantom{{}={}}
-\tfrac12f^{-1}\widetilde\theta'^2\Theta 
+\tfrac12\widetilde\theta'\widetilde\theta''\Theta 
+2\lambda\mu r_2\widetilde\theta'^{-1}\Theta 
-CQ
. 
\end{split}
\label{26060717l}
\end{equation}
Let us eliminate $\mathop{\mathrm{Re}}(\Theta' A)$ from the second and the third terms 
of \eqref{26060717l} by using \eqref{26052223}. In fact, we can rewrite and estimate them as 
\begin{equation}
\begin{split}
&
\lambda\mu\widetilde\theta'^{-1/2}\mathop{\mathrm{Re}}(\Theta' A)\widetilde\theta'^{-1/2}
-\lambda\mu f^{-1/2}\widetilde\theta'^{-1}\mathop{\mathrm{Re}}(\Theta' A)f^{-1/2}\widetilde\theta'^{-1}
\\&
\ge 
4\lambda^2\mu^2\widetilde\theta'^{-1}\Theta
-4\lambda^2\mu^2 f^{-1}\widetilde\theta'^{-2}\Theta
-2\lambda\mu r_2\widetilde\theta'^{-1}\Theta 
+2\lambda^2\mu^2\widetilde\theta'^{-3}\widetilde\theta''\Theta
-CQ
.
\end{split}
\label{26061223l}
\end{equation}
The estimates \eqref{26060717l} and \eqref{26061223l} imply 
\begin{equation}
\begin{split}
&A\bigl(\widetilde\theta'-\tfrac12f^{-1}\bigr)\Theta A
-\mathop{\mathrm{Re}}\bigl((2\lambda\mu-r_2)\Theta A\bigr)
-\widetilde\theta'^3\Theta 
\\&
\ge 
B^*\bigl(\widetilde\theta'-\tfrac12f^{-1}\bigr)\Theta B
-\tfrac12f^{-1}\widetilde\theta'^2\Theta
-2\lambda^2\mu^2f^{-1}\widetilde\theta'^{-2}\Theta
\\&\phantom{{}={}}
+\tfrac12\widetilde\theta'\widetilde\theta''\bigl(1+4\lambda^2\mu^2\widetilde\theta'^{-4}\bigr)\Theta
-CQ
. 
\end{split}
\label{26052423l}
\end{equation}
Thus by \eqref{26052317l}, \eqref{26052422l} and \eqref{26052423l} we obtain  
\begin{equation}
\begin{split}
\mathop{\mathrm{Im}}\bigl(A\Theta(H-\kappa^2)\bigr)
&\ge 
B^*\bigl(\widetilde\theta'-\tfrac18f^{-1}\bigr)\Theta B
+\tfrac12f^{-1}\bigl(1+2\lambda^2\widetilde\theta'^{-2}\bigr)\bigl(\widetilde\theta'^2-2\mu^2\bigr)\Theta
\\&\phantom{{}={}}
+\tfrac12\widetilde\theta'\widetilde\theta''\bigl(1+4\lambda^2\mu^2\widetilde\theta'^{-4}\bigr)\Theta
-CQ
.
\end{split}
\label{26052420l}
\end{equation}

Let us take a closer look at the second and the third terms of \eqref{26052420l}. 
Choose $s$ small enough, and restrict $m$ large enough depending on $t$ such that 
\[
\rho^{-1}(1-\rho)\sqrt2\mu\ge s\theta'+tf^{-\rho}\ \ \text{on }\mathop{\mathrm{supp}}\chi_{m,n}. 
\]
Then we have 
\[
\sqrt2\mu\ge \rho\widetilde\theta'\ \ \text{on }\mathop{\mathrm{supp}}\chi_{m,n}
,
\]
so that 
by using \eqref{2507191448d} and Lemma~\ref{2605242351} we can estimate them as 
\begin{equation}
\begin{split}
&
\tfrac12f^{-1}\bigl(1+2\lambda^2\widetilde\theta'^{-2}\bigr)\bigl(\widetilde\theta'^2-2\mu^2\bigr)\Theta
+\tfrac12\widetilde\theta'\widetilde\theta''\bigl(1+4\lambda^2\mu^2\widetilde\theta'^{-4}\bigr)\Theta
\\&\ge 
\tfrac{1+\rho}2f^{-1}\bigl(1+2\lambda^2\widetilde\theta'^{-2}\bigr)\bigl(s\theta'+(1-\rho)tf^{-\rho}\bigr)\widetilde\theta'\Theta
\\&\phantom{{}={}}
+\tfrac12f^{-1}\bigl(1+2\lambda^2\widetilde\theta'^{-2}\bigr)\bigl(sf\theta''-\rho(1-\rho)tf^{-\rho}\bigr)\widetilde\theta'\Theta
-CQ
\\&\ge 
C^{-1}tf^{-1-\rho}\Theta
-CQ
.
\end{split}
\label{260611123}
\end{equation}
We substitute \eqref{260611123} into \eqref{26052420l},
and also estimate $Q$ similarly to \eqref{26060922e}. 
At last we obtain 
\begin{equation*}
\begin{split}
\mathop{\mathrm{Im}}\bigl(A\Theta(H-\kappa^2)\bigr)
&\ge 
B^*\bigl(\widetilde\theta'-\tfrac12f^{-1}\bigr)\Theta B
+\bigl(C^{-1}tf^{-1-\rho}-Cf^{-1-\rho}-t^2f^{-1-2\rho}\bigr)\Theta
\\&\phantom{{}={}}
-Ct^2\bigl(\chi_{m-1,m+1}^2+\chi_{n-1,n+1}^2\bigr)f^{-1}\mathrm e^{2\widetilde\theta}
\\&\phantom{{}={}}
-C(H-\kappa^2)^*\chi_{m-1,n+1}f^{1+\rho}\mathrm e^{2\widetilde\theta}(H-\kappa^2)
.
\end{split}
%\label{26060813l}
\end{equation*}
Thus by letting $t$ large enough and then $m\in\mathbb N$ even larger if necessary 
we obtain the assertion. 
\end{proof}

\begin{proof}[Proof of Proposition~\ref{20260535017a} for non-real spectral parameters]
Fix any $\kappa=\lambda+\mathrm i\mu\in \mathbb C_+\setminus \mathbb R$, 
and choose $s\in(0,1]$ and $t\ge 1$ as in Lemma~\ref{260612}. 
Let $\phi\in  \mathcal E_{-\sqrt2\mu,-t}$ be a solution to \eqref{200125}. 
Then by using Lemma~\ref{260612} 
we can discuss similarly to the proof of Proposition~\ref{26060301a} to obtain
\[
\bigl\|\bar\chi_{m}\mathrm e^{(\sqrt2\mu+s)f+tf^{1-\rho}}\phi\bigr\|_{L^2}^2
\le 
C\bigl\|\chi_{m-1,m+1}\phi\bigr\|_{L^2}^2
.
\]
Thus we are done. 
\end{proof}

\subsubsection{Non-zero real spectral parameters}

Next we prove Proposition~\ref{20260535017a} 
for non-zero real spectral parameters. 
Let 
\begin{equation*}
\Theta=\chi_{m,n}\mathrm e^{2\widetilde\theta}
;\quad 
\widetilde\theta=s\theta+tf^{1-\rho}
,
%\label{2507191448}
\end{equation*}
depending on parameters $m,n,\nu\in\mathbb N$, $s\in (0,1]$ and $t\ge 1$ with $n>m$. 
Recall that $\chi_{m,n}$ and $\theta$ are from \eqref{eq:11.7.11.5.14} and Definition~\ref{25071914480}, respectively.

\begin{lemma}\label{26060317}
Suppose Assumption~\ref{cond:20022317} with (c),
and let  $\kappa=\lambda\in\mathbb R\setminus\{0\}$. 
Then there exist $c,C> 0$, $m\in\mathbb N$, $s\in(0,1]$ and $t\ge 1$ such that 
uniformly in $n>m$ and $\nu\in\mathbb N$
\[
\begin{split}
\mathop{\mathrm{Im}}\bigl(A\Theta(H-\kappa^2)\bigr)
&\ge cf^{-1}\Theta-C\bigl(\chi_{m-1,m+1}^2+\chi_{n-1,n+1}^2\bigr)f^{-1+\rho}\mathrm e^{2\widetilde\theta}
\\&\phantom{{}={}}
-C(H-\kappa^2)^*\chi_{m-1,n+1}f^{1+\rho}\mathrm e^{2\widetilde\theta}(H-\kappa^2)
.
\end{split}
\]  
\end{lemma}
\begin{proof}
Let $\kappa=\lambda\in\mathbb R\setminus\{0\}$ be fixed. 
For the moment do not the below estimates are uniform in $n>m\ge 1$, 
$\nu\in\mathbb N$, $s\in(0,1]$ and $t\ge 1$, 
and constants $C>0$ are independent of them. 

In this proof admissible error terms are of the form 
\begin{equation*}
\begin{split}
Q&=
t^2f^{-\min\{1+\rho,2-\rho\}}\Theta
+t^2\bigl(f^\rho|\chi_{m,n}'|+|\chi_{m,n}''|+|\chi_{m,n}'''|\bigr)\mathrm e^{2\widetilde\theta}
\\&\phantom{{}={}}
+p_j\bigl(f^{-\min\{1+\rho,2-\rho\}}\Theta+f^\rho|\chi_{m,n}'|\mathrm e^{2\widetilde\theta}\bigr) p_j
\\&\phantom{{}={}}
+(H-\kappa^2)^*\bigl(f^{1+\rho}\Theta+|\chi_{m,n}'|\mathrm e^{2\widetilde\theta}\bigr)(H-\kappa^2)
.
\end{split}
%\label{260524}
\end{equation*}
Similarly to \eqref{26052317l}, by \eqref{26052315} we have 
\begin{equation}
\begin{split}
\mathop{\mathrm{Im}}\bigl(A\Theta(H-\kappa^2)\bigr)
&=
A\widetilde\theta'\Theta A
+\tfrac12f^{-1}\Theta L
+\mathop{\mathrm{Re}}(r_1\Theta A)
\\&\phantom{{}={}}
-\widetilde\theta'^3 \Theta
-\tfrac13\widetilde\theta'\widetilde\theta'' \Theta
-CQ
\\&\ge 
\tfrac12Af^{-1}\Theta A
+\tfrac12f^{-1}\Theta L
+A\bigl(\widetilde\theta'-\tfrac12f^{-1}\bigr)\Theta A
\\&\phantom{{}={}}
+\mathop{\mathrm{Re}}(r_1\Theta A)
-\widetilde\theta'^3 \Theta
-\tfrac13\widetilde\theta'\widetilde\theta'' \Theta
-CQ
.
\end{split}
\label{26052317}
\end{equation}
Similarly to \eqref{26052422l},
we combine and estimate the first and the second terms on the right-hand side 
of \eqref{26052317} by using \eqref{26052222} as 
\begin{equation}
\begin{split}
\tfrac12Af^{-1}\Theta A
+\tfrac12f^{-1}\Theta L
&\ge 
\lambda^2f^{-1}\Theta 
+f^{-1}\widetilde\theta'^2\Theta
-CQ
.
\end{split}
\label{26052422}
\end{equation}
On the other hand, 
if we set 
\[
B=A+\tfrac12r_1\widetilde\theta'^{-1}-\mathrm i\widetilde\theta',
\]
then we rewrite and estimate the third to the sixth terms 
of \eqref{26052317} similarly to \eqref{26060717l} as 
\begin{equation}
\begin{split}
&
A\bigl(\widetilde\theta'-\tfrac12f^{-1}\bigr)\Theta A
+\mathop{\mathrm{Re}}(r_1\Theta A)
-\widetilde\theta'^3\Theta
-\tfrac32\widetilde\theta'\widetilde\theta''\Theta
\\&=
B^*\bigl(\widetilde\theta'-\tfrac12f^{-1}\bigr)\Theta B
+\tfrac12\mathop{\mathrm{Re}}\bigl(r_1f^{-1}\widetilde\theta'^{-1}\Theta A\bigr)
+\bigl(\widetilde\theta'(\widetilde\theta'-\tfrac12f^{-1})\Theta\bigr)' 
\\&\phantom{{}={}}
-\bigl(\tfrac14r_1^2\widetilde\theta'^{-2}+\widetilde\theta'^2\bigr)\bigl(\widetilde\theta'-\tfrac12f^{-1}\bigr)\Theta 
-\widetilde\theta'^3\Theta
-\tfrac32\widetilde\theta'\widetilde\theta''\Theta
\\&\ge
B^*\bigl(\widetilde\theta'-\tfrac12f^{-1}\bigr)\Theta B
+\tfrac12\mathop{\mathrm{Re}}\bigl(r_1f^{-1}\widetilde\theta'^{-1}\Theta A\bigr)
\\&\phantom{{}={}}
-\tfrac12f^{-1}\widetilde\theta'^2\Theta 
+\tfrac12\widetilde\theta'\widetilde\theta''\Theta
-\tfrac14r_1^2\widetilde\theta'^{-1}\Theta
-CQ
,
\end{split}
\label{26052423}
\end{equation}
where we have used 
\begin{equation}
0\le \widetilde\theta'^{-1}\le Ct^{-1}f^{-\rho}
.
\label{26061721}
\end{equation}
We rewrite the second term of \eqref{26052423} as follows. 
Let $M\ge 1$ be sufficiently large, and then partially similarly to \eqref{26061223l} 
by using \eqref{26052223} and \eqref{26061721}  
\begin{equation}
\begin{split}
&\tfrac12\mathop{\mathrm{Re}}\bigl(r_1f^{-1}\widetilde\theta'^{-1}\Theta A\bigr)
\\&
\ge 
\tfrac14\bigl(M\widetilde\theta'^{-1}+r_1f^{-1}\widetilde\theta'^{-2}\bigr)^{1/2}\mathop{\mathrm{Re}}(\Theta' A)\bigl(M\widetilde\theta'^{-1}+r_1f^{-1}\widetilde\theta'^{-2}\bigr)^{1/2}
\\&\phantom{{}={}}
-\tfrac14M\widetilde\theta'^{-1/2}\mathop{\mathrm{Re}}(\Theta' A)\widetilde\theta'^{-1/2}
-CQ
\\&
\ge 
\tfrac12\bigl(M\widetilde\theta'^{-1}+r_1f^{-1}\widetilde\theta'^{-2}\bigr)^{1/2}\mathop{\mathrm{Im}}\bigl(\Theta(H-\kappa^2)\bigr)\bigl(M\widetilde\theta'^{-1}+r_1f^{-1}\widetilde\theta'^{-2}\bigr)^{1/2}
\\&\phantom{{}={}}
-\tfrac12M\widetilde\theta'^{-1/2}\mathop{\mathrm{Im}}\bigl(\Theta(H-\kappa^2)\bigr)\widetilde\theta'^{-1/2}
-CQ
\\&
\ge 
-\tfrac14\mathop{\mathrm{Re}}\bigl((\partial_jr_1f^{-1}\widetilde\theta'^{-2})\Theta p_j\bigr)
-CQ
\\&
\ge 
-CQ.
\end{split}
\label{26061223}
\end{equation}
Thus by \eqref{26052423} and \eqref{26061223} 
\begin{equation}
\begin{split}
&A\bigl(\widetilde\theta'-\tfrac12f^{-1}\bigr)\Theta A
+\mathop{\mathrm{Re}}(r_1\Theta A)
-\widetilde\theta'^3\Theta
-\tfrac32\widetilde\theta'\widetilde\theta''\Theta
\\&\ge
B^*\bigl(\widetilde\theta'-\tfrac12f^{-1}\bigr)\Theta B
-\tfrac12f^{-1}\widetilde\theta'^2\Theta 
+\tfrac12\widetilde\theta'\widetilde\theta''\Theta
-\tfrac14r_1^2\widetilde\theta'^{-1}\Theta
-CQ.
\end{split}
\label{260617}
\end{equation}
Thus \eqref{26052317}, \eqref{26052422} and \eqref{260617} imply 
\begin{equation*}
\begin{split}
\mathop{\mathrm{Im}}\bigl(A\Theta(H-\kappa^2)\bigr)
&\ge 
B^*\bigl(\widetilde\theta'-\tfrac12f^{-1}\bigr)\Theta B
+\lambda^2f^{-1}\Theta 
+\tfrac12f^{-1}\widetilde\theta'^2\Theta 
+\tfrac12\widetilde\theta'\widetilde\theta''\Theta
\\&\phantom{{}={}}
-\tfrac14r_1^2\widetilde\theta'^{-1}\Theta
-CQ.
\end{split}
%\label{26052420}
\end{equation*}
Finally let us estimate $Q$ similarly to \eqref{26060922e}, and then 
\begin{equation}
\begin{split}
&\mathop{\mathrm{Im}}\bigl(A\Theta(H-\kappa^2)\bigr)
\\&\ge 
B^*\bigl(\widetilde\theta'-\tfrac12f^{-1}\bigr)\Theta B
+\bigl(\lambda^2+\tfrac12s(\theta'+f\theta'')\widetilde\theta'-\tfrac14r_1^2f\widetilde\theta'^{-1}\bigr) f^{-1}\Theta 
\\&\phantom{{}={}}
-Ct^2f^{-\max\{1+\rho,2-\rho\}}\Theta
-Ct^2\bigl(\chi_{m-1,m+1}^2+\chi_{n-1,n+1}^2\bigr)f^{-1+\rho}\mathrm e^{2\widetilde\theta}
\\&\phantom{{}={}}
-C(H-\kappa^2)^*\chi_{m-1,n+1}f^{1+\rho}\mathrm e^{2\widetilde\theta}(H-\kappa^2)
.
\end{split}
\label{2605242358}
\end{equation}
Now we choose small $s\in(0,1]$ and large $t\ge 1$, noting \eqref{26061721}. 
Then let $m\in\mathbb N$ be large enough, and we obtain the assertion. 
\end{proof}

\begin{proof}[Proof of Proposition~\ref{20260535017a} for non-zero real spectral parameters]
Fix any $\kappa=\lambda\in\mathbb R\setminus\{0\}$, 
and choose $m\in\mathbb N$, $s\in(0,1]$ and $t\ge 1$ as in Lemma~\ref{26060317}. 
Let $\phi\in  \mathcal E_{0,-t}$ be a solution to \eqref{200125}. 
Then, similarly to the proof of Proposition~\ref{26060301a}, 
we can verify 
\[
\bigl\|\bar\chi_{m}f^{-1/2}\mathrm e^{sf+tf^{1-\rho}}\phi\bigr\|_{L^2}^2
\le 
C\bigl\|\chi_{m-1,m+1}\phi\bigr\|_{L^2}^2
.
\]
This implies the assertion. 
\end{proof}

\subsection{Eigenfunctions with additional exponential decay}\label{2606071}

\subsubsection{Super-exponential decay estimates}

Here we prove the assertion 1 of Proposition~\ref{20260535017}.  
We let 
\begin{equation*}
\Theta=\chi_{m,n}\mathrm e^{2\widetilde\theta}
;\quad 
\widetilde\theta=t f+s\theta
,
%\label{2507191448i}
\end{equation*}
with parameters $m,n,\nu\in\mathbb N$, $t>\sqrt2\mu$ and  $s\in (0,1]$ satisfying $n>m$,
where 
$\chi_{m,n}$ and $\theta$ are from \eqref{eq:11.7.11.5.14} and Definition~\ref{25071914480}, respectively. 

\begin{lemma}\label{26060317i}
Suppose Assumption~\ref{cond:20022317} with (b), and let $\kappa=\lambda+\mathrm i\mu\in\mathbb C_+$ and $t_1>t_0>\sqrt2\mu$. 
Then there exist $c,C$, $m\in\mathbb N$ and $s\in (0,1]$ such that 
uniformly in $n>m$, $\nu\in\mathbb N$ and $t\in[t_0,t_1]$
\[
\begin{split}
\mathop{\mathrm{Im}}\bigl(A\Theta(H-\kappa^2)\bigr)
&\ge cf^{-1}\Theta-C\bigl(\chi_{m-1,m+1}^2+\chi_{n-1,n+1}^2\bigr)f^{-1}\mathrm e^{2\widetilde\theta}
\\&\phantom{{}={}}
-C(H-\kappa^2)^*\chi_{m-1,n+1}f^{1+\rho}\mathrm e^{2\widetilde\theta}(H-\kappa^2)
.
\end{split}
\]  
\end{lemma}
\begin{proof}
Let $\kappa=\lambda+\mathrm i\mu\in\mathbb C_+$ and $t_1>t_0>\sqrt2\mu$ be fixed. 
For the moment the below estimates are uniform in $n>m\ge 1$, $\nu\in\mathbb N$,
$t\in[t_0,t_1]$ and $s\in (0,1]$, 
and constants $C>0$ are independent of them. 
The admissible error terms are of the form
\begin{equation*}
\begin{split}
Q&=
sf^{-1}\Theta+f^{-1-\rho}\Theta
+\bigl(|\chi_{m,n}'|+|\chi_{m,n}''|+|\chi_{m,n}'''|\bigr)\mathrm e^{2\widetilde\theta}
\\&\phantom{{}={}}
+p_j\bigl(sf^{-1}\Theta+f^{-1-\rho}\Theta+|\chi_{m,n}'|\mathrm e^{2\widetilde\theta}\bigr) p_j
\\&\phantom{{}={}}
+(H-\kappa^2)^*\bigl(f^{1+\rho}\Theta+|\chi_{m,n}'|\mathrm e^{2\widetilde\theta}\bigr)(H-\kappa^2)
,
\end{split}
%\label{260524i}
\end{equation*}
depending on $s\in (0,1]$. 
Note that here the terms of the form $sf^{-1}\Theta$ are negligible, cf.\ \eqref{260524l}.
Then by the expression \eqref{26052315} we have, similarly to \eqref{26052317l}, 
\begin{equation}
\begin{split}
\mathop{\mathrm{Im}}\bigl(A\Theta(H-\kappa^2)\bigr)
&\ge 
\tfrac12Af^{-1}\Theta A
+\tfrac12f^{-1}\Theta L
+A\bigl(\widetilde\theta'-\tfrac12f^{-1}\bigr)\Theta A
\\&\phantom{{}={}}
-\mathop{\mathrm{Re}}\bigl((2\lambda\mu-r_1)\Theta A\bigr)
-\widetilde\theta'^3\Theta 
-CQ
.
\end{split}
\label{26052317i}
\end{equation}
Similarly to \eqref{26052422l}, the first and the second terms on the right-hand side 
of \eqref{26052317i} can be combined by using \eqref{26052222} as 
\begin{equation}
\begin{split}
\tfrac12Af^{-1}\Theta A
+\tfrac12f^{-1}\Theta L
%&=
%\tfrac14(f^{-1}\Theta)''
%+\bigl(\lambda^2-\mu^2-q_1-q_2\bigr)f^{-1}\Theta 
%\\&\phantom{{}={}}
%+\mathop{\mathrm{Re}}\bigl(f^{-1}\Theta (H-\kappa^2)\bigr)
%\\
&\ge 
f^{-1}\widetilde\theta'^2\Theta
+(\lambda^2-\mu^2)f^{-1}\Theta 
%+\lambda^2\Theta 
-CQ
.
\end{split}
\label{26052422i}
\end{equation}
The third to the fifth terms of \eqref{26052317i} 
can be rewritten and estimated similarly to \eqref{26052423l}. 
We first set 
\[
B=A-\tfrac12(4\lambda\mu-r_1)\widetilde\theta'^{-1}-\mathrm i\widetilde\theta',
\]
then the third to the fifth terms of \eqref{26052317i} are estimated as 
\begin{equation}
\begin{split}
&
A\bigl(\widetilde\theta'-\tfrac12f^{-1}\bigr)\Theta A
-\mathop{\mathrm{Re}}\bigl((2\lambda\mu-r_1)\Theta A\bigr)
-\widetilde\theta'^3\Theta 
\\&
\ge 
B^*\bigl(\widetilde\theta'-\tfrac12f^{-1}\bigr)\Theta B
+\lambda\mu\widetilde\theta'^{-1/2}\mathop{\mathrm{Re}}(\Theta' A)\widetilde\theta'^{-1/2}
\\&\phantom{{}={}}
-\lambda\mu f^{-1/2}\widetilde\theta'^{-1}\mathop{\mathrm{Re}}(\Theta' A)f^{-1/2}\widetilde\theta'^{-1}
-4\lambda^2\mu^2\widetilde\theta'^{-1}\Theta 
+2\lambda^2\mu^2f^{-1}\widetilde\theta'^{-2}\Theta
\\&\phantom{{}={}}
-\tfrac12f^{-1}\widetilde\theta'^2\Theta 
+2\lambda\mu r_1\widetilde\theta'^{-1}\Theta 
-CQ
. 
\end{split}
\label{26060717i}
\end{equation}
Similarly to \eqref{26061223l}, 
the second and the third terms of \eqref{26060717i} 
are estimated by using \eqref{26052223} as 
\begin{equation*}
\begin{split}
&
\lambda\mu\widetilde\theta'^{-1/2}\mathop{\mathrm{Re}}(\Theta' A)\widetilde\theta'^{-1/2}
-\lambda\mu f^{-1/2}\widetilde\theta'^{-1}\mathop{\mathrm{Re}}(\Theta' A)f^{-1/2}\widetilde\theta'^{-1}
\\&
\ge 
4\lambda^2\mu^2\widetilde\theta'^{-1}\Theta
-2\lambda\mu r_1\widetilde\theta'^{-1}\Theta 
-4\lambda^2\mu^2 f^{-1}\widetilde\theta'^{-2}\Theta 
-CQ
,
\end{split}
\end{equation*}
and, substitute the right above inequality into \eqref{26060717i}, and we obtain 
\begin{equation}
\begin{split}
&
A\bigl(\widetilde\theta'-\tfrac12f^{-1}\bigr)\Theta A
-\mathop{\mathrm{Re}}\bigl((2\lambda\mu-r_1)\Theta A\bigr)
-\widetilde\theta'^3\Theta 
\\&
\ge 
B^*\bigl(\widetilde\theta'-\tfrac12f^{-1}\bigr)\Theta B
-\tfrac12f^{-1}\widetilde\theta'^2\Theta 
-2\lambda^2\mu^2f^{-1}\widetilde\theta'^{-2}\Theta
-CQ
.
\end{split}
\label{26052423i}
\end{equation}
Thus \eqref{26052317i}, \eqref{26052422i} and \eqref{26052423i} imply 
\begin{equation*}
\begin{split}
\mathop{\mathrm{Im}}\bigl(A\Theta(H-\kappa^2)\bigr)
&\ge 
B^*\bigl(\widetilde\theta'-\tfrac12f^{-1}\bigr)\Theta B
+\tfrac12f^{-1}\widetilde\theta'^2\bigl(1+2\lambda^2\widetilde\theta'^{-2}\bigr)\bigl(1-2\mu^2\widetilde\theta'^{-2}\bigr)\Theta 
\\&\phantom{{}={}}
-CQ
.
\end{split}
%\label{26052420i}
\end{equation*}
Finally we estimate $Q$ similarly to \eqref{26060922e}, 
and then obtain 
\begin{equation}
\begin{split}
&\mathop{\mathrm{Im}}\bigl(A\Theta(H-\kappa^2)\bigr)
\\&\ge 
B^*\bigl(\widetilde\theta'-\tfrac12f^{-1}\bigr)\Theta B
+\tfrac12f^{-1}\widetilde\theta'^2\bigl(1+2\lambda^2\widetilde\theta'^{-2}\bigr)\bigl(1-2\mu^2\widetilde\theta'^{-2}\bigr)\Theta 
\\&\phantom{{}={}}
-Csf^{-1}\Theta
-Cf^{-1-\rho}\Theta
-C\bigl(\chi_{m-1,m+1}^2+\chi_{n-1,n+1}^2\bigr)f^{-1}\mathrm e^{2\widetilde\theta}
\\&\phantom{{}={}}
-C(H-\kappa^2)^*\chi_{m-1,n+1}f^{1+\rho}\mathrm e^{2\widetilde\theta}(H-\kappa^2)
.
\end{split}
\label{26060813i}
\end{equation}
Thus by letting $m\in\mathbb N$ be large enough and $s\in(0,1]$ be small enough, 
we obtain the assertion. 
\end{proof}

\begin{proof}[Proof of the assertion 1 of Proposition~\ref{20260535017}.]
Let $\phi\in \mathcal E_{-\sqrt2\mu-0}$ be a solution to \eqref{200125} 
with $\kappa=\lambda+\mathrm i\mu\in \mathbb C_+$. 
Let 
\[
t_1=\sup\bigl\{t\in\mathbb R;\ \phi\in \mathcal E_{-t} \bigr\}>\sqrt2\mu
,
\]
and assume $t_1\neq \infty$. We shall deduce a contradiction. 
Take any $t_0\in (\sqrt2\mu,t_1)$, choose $s_0\in(0,1]$ and 
$m\in\mathbb N$ as in Lemma~\ref{26060317i}, 
and then fix $t\in [t_0,t_1)$ such that $t+s>t_1$. 
Then, similarly to the proof of Proposition~\ref{26060301a}, 
it follows by Lemma~\ref{26060317i} that 
%on the state $\chi_{m-2,n+2}\phi$, and then we have for any $n>m$ and $\nu\in\mathbb N$
%\begin{equation*}
%\bigl\|\chi_{m,n}f^{-1/2}\mathrm e^{t f+s\theta}\phi\bigr\|_{L^2}^2
%\le 
%C\bigl\|\chi_{m-1,m+1}\phi\bigr\|_{L^2}^2
%+C_\nu\bigl\|\chi_{n-1,n+1}f^{-1/2}\mathrm e^{t f}\phi\bigr\|_{L^2}^2
%\label{2606031731i}
%\end{equation*}
%with $C_\nu$ independent of $n$. 
%Let $n\to\infty$, and then let $\nu\to\infty$, and it follows that 
\begin{equation*}
\bigl\|\bar\chi_{m}f^{-1/2}\mathrm e^{(t+s)f}\phi\bigr\|_{L^2}^2
\le 
C\bigl\|\chi_{m-1,m+1}\phi\bigr\|_{L^2}^2
.
\end{equation*}
This implies $\phi\in \mathcal E_{-r}$ for any $r<t+s$, which is a contradiction. 
Thus we are done. 
\end{proof}

\subsubsection{Absence of super-exponentially decaying eigenfunctions}

Next we complete the proof of Proposition~\ref{20260535017} by showing 
the assertion 2. 
Here we choose a simple weight function of the form
\begin{equation*}
\Theta=\chi_{m,n}\mathrm e^{2sf}
%\label{2507191448b}
\end{equation*}
with parameters $m,n\in\mathbb N$ and  $s\in\mathbb R$ satisfying $n>m$,
where $\chi_{m,n}$ and $\theta$ 
are from \eqref{eq:11.7.11.5.14} and Definition~\ref{25071914480}, respectively. 
In the below estimate we particularly focus on uniformity in large $s>\sqrt2\mu$.

\begin{lemma}\label{20260535018}
Suppose Assumption~\ref{cond:20022317} with (b),
and let  $\kappa=\lambda+\mathrm i\mu\in\mathbb C_+$ and $s_0>\sqrt2\mu$. 
Then there exist $c,C> 0$ and $m\in\mathbb N$ such that 
uniformly in $n>m$ and $s\ge s_0$
\[
\begin{split}
\mathop{\mathrm{Im}}\bigl(A\Theta(H-\kappa^2)\bigr)
&\ge cs^2f^{-1}\Theta-Cs^2\bigl(\chi_{m-1,m+1}^2+\chi_{n-1,n+1}^2\bigr)f^{-1}\mathrm e^{2sf}
\\&\phantom{{}={}}
-C(H-\kappa^2)^*\chi_{m-1,n+1}f^{1+\rho}\mathrm e^{2sf}(H-\kappa^2)
.
\end{split}
\]  
\end{lemma}
\begin{proof}
Fix any $\kappa=\lambda+\mathrm i\mu\in \mathbb C_+$ and $s_0>\sqrt2\mu$. 
For the moment the below estimates are uniform in $n>m\ge 1$ and $s\ge s_0$, 
and constants $C>0$ are independent of them. 

Here we choose admissible error terms of the form 
\begin{equation*}
\begin{split}
Q&=
s^2f^{-1-\rho}\Theta
+s^2\bigl(|\chi_{m,n}'|+|\chi_{m,n}''|+|\chi_{m,n}'''|\bigr)\mathrm e^{2sf}
+p_j\bigl(f^{-1-\rho}\Theta+|\chi_{m,n}'|\mathrm e^{2sf}\bigr) p_j
\\&\phantom{{}={}}
+(H-\kappa^2)^*\bigl(f^{1+\rho}\Theta+|\chi_{m,n}'|\mathrm e^{2sf}\bigr)(H-\kappa^2)
, 
\end{split}
%\label{260524b}
\end{equation*}
which depends on $s\ge s_0$. 
Then by the expression \eqref{26052315} we have, similarly to \eqref{26052317i},  
\begin{equation}
\begin{split}
\mathop{\mathrm{Im}}\bigl(A\Theta(H-\kappa^2)\bigr)
&\ge 
%sA\Theta A
%+\tfrac12f^{-1}\Theta L
%-2\lambda\mu\mathop{\mathrm{Re}}(\Theta A)
%-s^3\Theta
%\\&\phantom{{}={}}
%+\mathop{\mathrm{Re}}(r_1\Theta A)
%-CQ
%\\&=
\tfrac12Af^{-1}\Theta A
+\tfrac12f^{-1}\Theta L
+A\big(s-\tfrac12f^{-1}\bigr)\Theta A
\\&\phantom{{}={}}
-\mathop{\mathrm{Re}}\bigl((2\lambda\mu-r_1)\Theta A\bigr)
-s^3\Theta
-CQ
.
\end{split}
\label{26052317b}
\end{equation}
The first and the second terms on the right-hand side 
of \eqref{26052317b} can be combined by using \eqref{26052222}, similarly to \eqref{26052422i}, as 
\begin{equation}
\begin{split}
\tfrac12Af^{-1}\Theta A
+\tfrac12f^{-1}\Theta L
&%=
%\tfrac14(f^{-1}\Theta)''
%+\tfrac12\mathop{\mathrm{Re}}\bigl(f^{-1}\Theta (A^2+L)\bigr)
%\\&\ge 
%s^2f^{-1}\Theta
%+\bigl(\lambda^2-\mu^2-q_1-q_2\bigr)f^{-1}\Theta 
%\\&\phantom{{}={}}
%+\mathop{\mathrm{Re}}\bigl(f^{-1}\Theta (H-\kappa^2)\bigr)
%-CQ
%\\&
\ge 
s^2f^{-1}\Theta
+(\lambda^2-\mu^2)f^{-1}\Theta 
-CQ
.
\end{split}
\label{26052422b}
\end{equation}
The third to fifth terms of \eqref{26052317b} are rewritten similarly to \eqref{26060717i}. 
In fact, if we set 
\[
B=A-\tfrac12(4\lambda\mu-r_1)s^{-1}-\mathrm is,
\]
then after some computations 
\begin{equation}
\begin{split}
&
A\big(s-\tfrac12f^{-1}\bigr)\Theta A
-2\lambda\mu\mathop{\mathrm{Re}}(\Theta A)
-s^3\Theta
+\mathop{\mathrm{Re}}(r_1\Theta A)
%\\&=
%\bigl(A-(2\lambda\mu-\tfrac12r_1) s^{-1}+\mathrm is\bigr)\big(s-\tfrac12f^{-1}\bigr)\Theta \bigl(A-(2\lambda\mu-\tfrac12r_1) s^{-1}-\mathrm is\bigr)
%\\&\phantom{{}={}}
%+2\lambda\mu \mathop{\mathrm{Re}}(\Theta A)
%-2\lambda\mu s^{-1}\mathop{\mathrm{Re}}(f^{-1}\Theta A)
%+\tfrac12 s^{-1}\mathop{\mathrm{Re}}(f^{-1}r_1\Theta A)
%\\&\phantom{{}={}}
%+s\bigl((s-\tfrac12f^{-1})\Theta \bigr)'
%-\bigl((2\lambda\mu-\tfrac12r_1)^2 s^{-2}+s^2\bigr)\big(s-\tfrac12f^{-1}\bigr)\Theta 
%-s^3\Theta
\\&
\ge 
B^*\big(s-\tfrac12f^{-1}\bigr)\Theta B
+\lambda\mu s^{-1}\mathop{\mathrm{Re}}(\Theta' A)
-\lambda\mu s^{-2}f^{-1/2}\mathop{\mathrm{Re}}(\Theta' A)f^{-1/2}
\\&\phantom{{}={}}
-4\lambda^2\mu^2 s^{-1}\Theta 
+2r_1\lambda\mu s^{-1}\Theta 
+2\lambda^2\mu^2s^{-2}f^{-1}\Theta 
-\tfrac12s^2f^{-1}\Theta 
-CQ
.
\end{split}
\label{26060717b}
\end{equation}
We can remove $\mathop{\mathrm{Re}}(\Theta' A)$ in the second and the third terms of 
\eqref{26060717b} by using \eqref{26052223}. At last, similarly to \eqref{26052423l}, 
we obtain
\begin{equation}
\begin{split}
&
A\big(s-\tfrac12f^{-1}\bigr)\Theta A
-2\lambda\mu\mathop{\mathrm{Re}}(\Theta A)
-s^3\Theta
+\mathop{\mathrm{Re}}(r_1\Theta A)
\\&
\ge 
B^*\big(s-\tfrac12f^{-1}\bigr)\Theta B
-\tfrac12s^2f^{-1}\Theta 
-2\lambda^2\mu^2s^{-2}f^{-1}\Theta 
-CQ
,
\end{split}
\label{26052423b}
\end{equation}
and by \eqref{26052317b}, \eqref{26052422b} and \eqref{26052423b} it follows that 
\begin{equation*}
\begin{split}
\mathop{\mathrm{Im}}\bigl(A\Theta(H-\kappa^2)\bigr)
&\ge 
B^*\big(s-\tfrac12f^{-1}\bigr)\Theta B
+\tfrac12s^2f^{-1}\bigl(1+2\lambda^2s^{-2}\bigr)\bigl(1-2\mu^2s^{-2}\bigr)\Theta 
\\&\phantom{{}={}}
-CQ
.
\end{split}
%\label{26052420b}
\end{equation*}
Lastly we estimate a contribution from $Q$ as in \eqref{26060922e}, 
and thus obtain 
\begin{equation*}
\begin{split}
\mathop{\mathrm{Im}}\bigl(A\Theta(H-\kappa^2)\bigr)
&\ge 
B^*\big(s-\tfrac12f^{-1}\bigr)\Theta B
+\tfrac12s^2f^{-1}\bigl(1+2\lambda^2s^{-2}\bigr)\bigl(1-2\mu^2s^{-2}\bigr)\Theta 
\\&\phantom{{}={}}
-Cs^2f^{-1-\rho}\Theta
-Cs^2\bigl(\chi_{m-1,m+1}^2+\chi_{n-1,n+1}^2\bigr)f^{-1}\mathrm e^{2s\theta}
\\&\phantom{{}={}}
-C(H-\kappa^2)^*\chi_{m-1,n+1}f^{1+\rho}\mathrm e^{2s\theta}(H-\kappa^2)
.
\end{split}
\end{equation*}
Now by letting $m\in\mathbb N$ be large enough we obtain the desired estimate. 
\end{proof}

\begin{proof}[Proof of the assertion 2 of Proposition~\ref{20260535017}.]
We discuss slightly differently from so far. 
Let $\phi\in\mathcal E_{-\infty}$ be a solution to \eqref{200125} 
with $\kappa=\lambda+\mathrm i\mu\in \mathbb C_+$. 
Fix any $s_0>\sqrt2\mu$, and take $m\in\mathbb  N$ as in Lemma~\ref{20260535018}. 
Then by taking an expectation value of the inequality from Lemma~\ref{20260535018}
on the state $\chi_{m-2,n+2}\phi$
we have for any $n>m$ and $s\ge s_0$
\begin{equation*}
\bigl\|\chi_{m,n}f^{-1/2}\mathrm e^{sf}\phi\bigr\|_{L^2}^2
\le 
C\bigl\|\chi_{m-1,m+1}f^{-1/2}\mathrm e^{sf}\phi\bigr\|_{L^2}^2
+C\bigl\|\chi_{n-1,n+1}f^{-1/2}\mathrm e^{s f}\phi\bigr\|_{L^2}^2
.
%\label{2606031731i}
\end{equation*}
If we let $n\to\infty$ and multiply the both sides by $\mathrm e^{-s2^{m+2}}$, it follows that 
\begin{equation}
\bigl\|\bar\chi_{m}f^{-1/2}\mathrm e^{s(f-2^{m+2})}\phi\bigr\|_{L^2}^2
\le 
C\bigl\|\chi_{m-1,m+1}f^{-1/2}\mathrm e^{s(f-2^{m+2})}\phi\bigr\|_{L^2}^2
.
\label{2606031731b}
\end{equation}
Now assume $\|\bar\chi_{m+2}\phi\|_{L^2}>0$. Then, as $s\to\infty$, 
the left-hand side of \eqref{2606031731b} grows exponentially while the right-hand side remains bounded. 
This is a contradiction, and it follows that $\bar\chi_{m+2}\phi=0$ on $\mathbb R^d$. 
Then by the unique continuation property we obtain 
$\phi=0$ on $\mathbb R^d$. Thus we are done. 
\end{proof}

\subsection{Rellich's theorem for zero eigenvalue}\label{2606072}

Finally we prove Proposition~\ref{20260535017b}. 
We can discuss both of the assertions 1 and 2 by using the following single lemma. 
Take a weight function of the form 
\begin{equation*}
\Theta=\chi_{m,n}\mathrm e^{2\widetilde\theta}
;\quad 
\widetilde\theta=t f^{1-\rho}+s\theta^{1-\rho}
,
%\label{2507191448m}
\end{equation*}
with parameters $m,n,\nu\in\mathbb N$, $t\ge 1$ and  $s\in [0,1]$ satisfying $n>m$,
where 
$\chi_{m,n}$ and $\theta$ are from \eqref{eq:11.7.11.5.14} and Definition~\ref{25071914480}, respectively. 

\begin{lemma}\label{26060317m}
Suppose Assumption~\ref{cond:20022317} with (d), let $\kappa=0$ and $s\in[0,1]$,
and let $t_0\ge 1$ be sufficiently large.
Then there exist $c,C$ and $m\in\mathbb N$ such that 
uniformly in $n>m$, $\nu\in\mathbb N$ and $t\ge t_0$
\[
\begin{split}
\mathop{\mathrm{Im}}\bigl(A\Theta(H-\kappa^2)\bigr)
&\ge ct^2f^{-1-2\rho}\Theta
-Ct^2\bigl(\chi_{m-1,m+1}^2+\chi_{n-1,n+1}^2\bigr)f^{-1+\rho}\mathrm e^{2\widetilde\theta}
\\&\phantom{{}={}}
-C(H-\kappa^2)^*\chi_{m-1,n+1}f^{1+2\rho}\mathrm e^{2\widetilde\theta}(H-\kappa^2)
.
\end{split}
\]  
\end{lemma}
\begin{proof}
Since outline of the proof is the same as that of Lemma~\ref{260612}, 
we do not explain the details, see the corresponding parts there. 
Nonetheless, we have to take care of difference of the settings, 
and in particular the dependence of the parameters. 
Hence let us record each of the key inequalities.

Let $\kappa=0$. 
For the moment the below estimates are uniform in $n>m\ge 1$, $\nu\in\mathbb N$,
$t\ge 1$ and $s\in (0,1]$, 
and constants $C>0$ are independent of them. 
The admissible error terms are of the form
\begin{equation*}
\begin{split}
Q&=
f^{-1-2\rho}\Theta
+t^2f^{-1-3\rho}\Theta
+t^2\bigl(f^\rho|\chi_{m,n}'|+|\chi_{m,n}''|+|\chi_{m,n}'''|\bigr)\mathrm e^{2\widetilde\theta}
\\&\phantom{{}={}}
+p_j\bigl(f^{-1-2\rho}\Theta+f^{\rho}|\chi_{m,n}'|\mathrm e^{2\widetilde\theta} \bigr)p_j
+(H-\kappa^2)^*\bigl(f^{1+2\rho}\Theta+|\chi_{m,n}'|\mathrm e^{2\widetilde\theta} \bigr)(H-\kappa^2)
.
\end{split}
%\label{260524m}
\end{equation*}
First by \eqref{26052315} we have%, similarly to \eqref{26052317l}, 
\begin{equation}
\begin{split}
\mathop{\mathrm{Im}}\bigl(A\Theta(H-\kappa^2)\bigr)
&\ge 
\tfrac12Af^{-1}\Theta A
+\tfrac12f^{-1}\Theta L
+A\bigl(\widetilde\theta'-\tfrac12f^{-1}\bigr)\Theta A
\\&\phantom{{}={}}
+\mathop{\mathrm{Re}}(r_1\Theta A)
-\widetilde\theta'^3\Theta 
-\tfrac32\widetilde\theta'\widetilde\theta''\Theta 
-CQ
.
\end{split}
\label{26052317m}
\end{equation}
%Similarly to \eqref{26052422l}, 
The first and the second terms of \eqref{26052317m} are combined by using \eqref{26052222} as 
\begin{equation}
\begin{split}
\tfrac12Af^{-1}\Theta A
+\tfrac12f^{-1}\Theta L
%&=
%\tfrac14(f^{-1}\Theta)''
%+\bigl(\lambda^2-\mu^2-q_1-q_2\bigr)f^{-1}\Theta 
%\\&\phantom{{}={}}
%+\mathop{\mathrm{Re}}\bigl(f^{-1}\Theta (H-\kappa^2)\bigr)
%\\
&\ge 
f^{-1}\widetilde\theta'^2\Theta
-CQ
.
\end{split}
\label{26052422m}
\end{equation}
Here we follow computations similar to \eqref{26052423}
to rewrite and estimate the third to the sixth terms of \eqref{26052317m} as 
\begin{equation}
\begin{split}
&A\bigl(\widetilde\theta'-\tfrac12f^{-1}\bigr)\Theta A
+\mathop{\mathrm{Re}}(r_1\Theta A)
-\widetilde\theta'^3\Theta 
-\tfrac32\widetilde\theta'\widetilde\theta''\Theta 
\\&
\ge 
\bigl(A+\mathrm i\widetilde\theta'\bigr)\bigl(\widetilde\theta'-\tfrac12f^{-1}\bigr)\Theta \bigl(A-\mathrm i\widetilde\theta'\bigr)
+\tfrac12\mathop{\mathrm{Re}}\bigl(r_1f^{-1}\widetilde\theta'^{-1}\Theta A\bigr)
\\&\phantom{{}={}}
-\tfrac12f^{-1}\widetilde\theta'^2\Theta 
+\tfrac12\widetilde\theta'\widetilde\theta''\Theta 
-CQ
. 
\end{split}
\label{26060717m}
\end{equation}
As for the second term of \eqref{26060717m}, we repeat the same computations 
as in \eqref{26061223} to obtain 
\begin{equation}
\begin{split}
\tfrac12\mathop{\mathrm{Re}}\bigl(r_1f^{-1}\widetilde\theta'^{-1}\Theta A\bigr)
\ge 
-CQ.
\end{split}
\label{26061223m}
\end{equation}
Thus \eqref{26052317m}, \eqref{26052422m}, \eqref{26060717m} and \eqref{26061223m} imply 
\begin{equation*}
\begin{split}
\mathop{\mathrm{Im}}\bigl(A\Theta(H-\kappa^2)\bigr)
&\ge 
\bigl(A+\mathrm i\widetilde\theta'\bigr)\bigl(\widetilde\theta'-\tfrac12f^{-1}\bigr)\Theta \bigl(A-\mathrm i\widetilde\theta'\bigr)
\\&\phantom{{}={}}
+\tfrac12f^{-1}\widetilde\theta'^2\Theta 
+\tfrac12\widetilde\theta'\widetilde\theta''\Theta 
-CQ
.
\end{split}
%\label{26052420m}
\end{equation*}
Now we take $Q$ into account, and obtain 
\begin{equation*}
\begin{split}
\mathop{\mathrm{Im}}\bigl(A\Theta(H-\kappa^2)\bigr)
&\ge 
\bigl(A+\mathrm i\widetilde\theta'\bigr)\bigl(\widetilde\theta'-\tfrac12f^{-1}\bigr)\Theta \bigl(A-\mathrm i\widetilde\theta'\bigr)
\\&\phantom{{}={}}
+\bigl(\tfrac12\widetilde\theta'((1-2\rho)^2t-s)-Cf^{-\rho}-Ct^2f^{-2\rho}\bigr)f^{-1-\rho}\Theta 
\\&\phantom{{}={}}
-Ct^2\bigl(\chi_{m-1,m+1}^2+\chi_{n-1,n+1}^2\bigr)f^{-1+\rho}\mathrm e^{2\widetilde\theta}
\\&\phantom{{}={}}
-C(H-\kappa^2)^*\chi_{m-1,n+1}f^{1+2\rho}\mathrm e^{2\widetilde\theta}(H-\kappa^2)
.
\end{split}
%\label{26060813m}
\end{equation*}
Thus, let $m\in\mathbb N$ and $t_0\ge 1$ be large enough,
and restrict $t\ge t_0$, 
and we obtain the assertion. 
\end{proof}

\begin{proof}[Proof of Proposition~\ref{20260535017b}]
\textit{1}.
Let $\kappa=0$ and $s=1$, and take $t_0\ge 1$ and $m\in\mathbb N$ as 
in Lemma~\ref{26060317m}. Let $\phi\in \mathcal E_{0,-t_0}$ solve \eqref{200125}, 
and set 
\[
t_1=\sup\bigl\{t\in\mathbb R;\ \phi\in \mathcal E_{0,-t} \bigr\}\ge t_0
.
\]
Assume $t_1\neq \infty$, and we shall deduce a contradiction. 
Fix $t\in [t_0,t_1)$ such that $t+1>t_1$. 
Then, similarly to the proof of the assertion 1 of Proposition~\ref{26060301a}, 
we can deduce from the inequality of Lemma~\ref{26060317m} that 
\begin{equation*}
\bigl\|\bar\chi_{m}f^{-1/2-\rho}\mathrm e^{(t+1)f^{1-\rho}}\phi\bigr\|_{L^2}^2
\le 
C\bigl\|\chi_{m-1,m+1}\phi\bigr\|_{L^2}^2
.
\end{equation*}
This implies $\phi\in \mathcal E_{0,-r}$ for any $r<t+1$, and hence a contradiction. 
Thus we are done. 

\smallskip
\noindent
\textit{2}. 
Next we let $s=0$ in Lemma~\ref{26060317m}, and then
we can discuss similarly to the proof of the assertion 2 of Proposition~\ref{26060301a}.
Since the arguments are completely parallel, we may omit the details. 
Then we are done. 
\end{proof}

\subsubsection*{Acknowledgements} 
KI was partially supported by JSPS KAKENHI Grant Number JP23K03163.

\providecommand{\bysame}{\leavevmode\hbox to3em{\hrulefill}\thinspace}
\providecommand{\MR}{\relax\ifhmode\unskip\space\fi MR }
\providecommand{\MRhref}[2]{%
  \href{http://www.ams.org/mathscinet-getitem?mr=#1}{#2}
}
\providecommand{\href}[2]{#2}

\end{document}